\documentclass{article}
\usepackage{amsfonts,latexsym,hyperref,amssymb,amsmath,amsthm} 
\newcounter{satznum}
\newtheorem{theorem}{Theorem}[satznum]
\newtheorem{definition}[theorem]{Definition}
\newtheorem{lemma}[theorem]{Lemma}
\newtheorem{corollary}[theorem]{Corollary}

\newtheorem{prop}[theorem]{Proposition}
\newtheorem{remark}[theorem]{Remark}
\newtheorem{example}[theorem]{Example}
\newenvironment{acknowledgement}
 {\begin{trivlist}\item[]{\bf Acknowledgement.}}
 {\end{trivlist}}

\newenvironment{satz-liste}
{\begin{list}{}{\setlength{\itemindent}{-0.2cm} \setlength{\itemsep}{0cm} \setlength{\topsep}{0cm} \setlength{\leftmargin}{0.5cm} \setlength{\labelwidth}{0.15cm} \setlength{\labelsep}{0.1cm} \setlength{\listparindent}{0cm} \setlength{\parsep}{0cm}} }
{\end{list}}

\newenvironment{text-liste}
{\begin{list}{}{\setlength{\itemindent}{-0.2cm} \setlength{\itemsep}{0.04cm} \setlength{\topsep}{0.1cm} \setlength{\leftmargin}{0.5cm} \setlength{\labelwidth}{0.2cm} \setlength{\labelsep}{0.1cm} \setlength{\listparindent}{0cm} \setlength{\parsep}{0.1cm}} }
{\end{list}}

\newenvironment{text-liste2}
{\begin{list}{}{\setlength{\itemindent}{-0.2cm} \setlength{\itemsep}{0.04cm} \setlength{\topsep}{0.1cm} \setlength{\leftmargin}{0.6cm} \setlength{\labelwidth}{0.2cm} \setlength{\labelsep}{0.1cm} \setlength{\listparindent}{0cm} \setlength{\parsep}{0.1cm}} }
{\end{list}}

{\makeatletter
\gdef\nz{{\mathbb N}} 
\gdef\R{{\mathbb R}} 

\gdef\om{{\omega}}
\gdef\Om{{\Omega}}

\gdef\N{{\mathbb N}}

}

\begin{document}
   \section*{Countable Random Sets: Uniqueness in Law and Constructiveness}
   {\sc Philip Herriger}\footnote{Mathematisches Institut, Eberhard Karls Universit\"at 
   T\"ubingen, Auf der Morgenstelle 10, 72076 T\"ubingen, Germany, E-mail address: 
   philip.herriger@uni-tuebingen.de}
\begin{abstract}
  The first part of this article deals with theorems on uniqueness in law for $\sigma$-finite and constructive countable random sets, which in contrast to the usual assumptions may have points of accumulation.  We discuss and compare two approaches on uniqueness theorems: First, the study of generators for $\sigma$-fields used in this context and, secondly, the analysis of hitting functions.
  
  The last section of this paper deals with the notion of constructiveness. We will prove a measurable selection theorem and a decomposition theorem for constructive countable random sets, and study constructive countable random sets with independent increments. 
   \vspace{2mm}

   \noindent Keywords: Constructive countability; constructiveness; countable random sets; decomposition; generators; hitting functions; independent increments; measurable selections; point processes; Poisson processes; R\'{e}nyi; uniqueness in law

   \vspace{2mm}

   \noindent 2010 Mathematics Subject Classification:
            Primary 60G55; 60D05   Secondary 28B20  
\end{abstract}
%


\subsection{Introduction} \label{intro}
\setcounter{theorem}{0}
	 Throughout this paper let $(S,\mathcal{S})$ be a measurable space, called {\it state space}, and let $(\Omega,\mathcal{F},P)$ denote a basic probability space. Following 
	 Kingman's approach \cite{kingman}, we define $C(S)$ to be the set of all countable (denumerable or finite) subsets of $S$ and denote  by $N_{A}$ the map $C(S)\to \N_{0}\cup\{\infty\}$, 
	 $M\mapsto |A\cap M|$ for all $A\subseteq S$.  Then a {\it countable random set} (cr-set) is a 
	 random variable $\pi: (\Omega,\mathcal{F})\to (C(S),\mathcal{C}(\mathcal{S}))$, 
	 in which $\mathcal{C}(\mathcal{S})$ is the smallest $\sigma$-field making $N_{A}$ for all $A\in\mathcal{S}$ measurable.
	 As in \cite{kingman}, the diagonal $\Delta:=\{(x,x)\,|\, x\in S\}$ of $S\times S$  is most of the time assumed to be measurable with respect to $\mathcal{S}\otimes\mathcal{S}$. We will enlarge upon this condition in Section 2, where we collect several tools used throughout the article.
	 
	 By definition of $\mathcal{C}(\mathcal{S})$, it follows immediately that for any cr-set
	 $\pi$ the map $\Om\times\mathcal{S}\to \N_{0}\cup\{\infty\}$, 
	 $(\om,A)\mapsto N_{A}(\pi(\om))$ is a kernel. Thus cr-sets are related to the 
	 theory of random measures. If $\Delta\in \mathcal{S}\otimes\mathcal{S}$ and if $\mu$ 
	 is a $\sigma$-finite measure on $(S,\mathcal{S})$ with values in $\N_{0}\cup\{\infty\}$, then there  exist Dirac measures $\delta_{x_{n}}$ with $x_{n}\in S$ and $\alpha_{n}\in\nz_{0}$ such that 
	 $\mu=\sum_{n\in\N}\alpha_{n}\delta_{x_{n}}$. Therefore any simple point process on 
	 $(S,\mathcal{S})$ can be regarded as a cr-set. For the definition of a simple point process or a kernel, see \cite{kallenberg}. However, $\mathcal{C}(\mathcal{S})$ can also be identified as the ``hit or miss'' $\sigma$-field on $C(S)$, a concept relating to the theory of random sets.
	 	
	 If $\pi$ is a cr-set, we denote its law on $(C(S),\mathcal{C}(\mathcal{S}))$ by $P_{\pi}$. The measure $\mu(A):=E(N_{A}(\pi))$ on $(S,\mathcal{S})$ is called 
	 the {\it intensity} of $\pi$. A map $\tau:\Om\rightarrow C(S)$ is said to be {\it 
	 finite} if $|\tau(\om)|<\infty$ for all $\om\in\Om$. Consequently, $\tau$ is called 
	 {\it $\sigma$-finite on}  $\mathcal{E}\subseteq \mathcal{S}$ if there is a covering 
	 $A_{n}\in\mathcal{E}$ with $S=\bigcup_{n\in\N}A_{n}$ such that $\tau\cap A_{n}$ is finite for all $n\in\N$. 
	 Note that there exists always a finite ($\sigma$-finite) version of a cr-set $\pi$ if the 
	 intensity measure of $\pi$ is finite ($\sigma$-finite). 

         Theorems on uniqueness in law for simple point processes usually assume a polish state space $(S,\mathcal{S})$ and finiteness on bounded sets (for a fixed metric), i.e. $|\pi(\om)\cap A|<\infty$ for all $\om\in\Om$ and all bounded sets $A\subseteq S$. However, it is possible to achieve good results by only working with $\sigma$-finiteness on arbitrary measurable spaces, as shown in the following theorem.

\begin{theorem} \label{sigma-theorem}
Let $\mathcal{E}\subseteq \mathcal{P}(S)$ be a $\cap$-stable generator of $\mathcal{S}$ and let $\pi_{1},\pi_{2}:\Om\to C(S)$ be two cr-sets. If $\pi_{1}$ and $\pi_{2}$ are $\sigma$-finite on $\mathcal{E}$ and satisfy
\begin{equation}\label{eq-theosigma} P_{\pi_{1}}(N_{A_{1}}=k_{1},\ldots, N_{A_{n}}=k_{n})=P_{\pi_{2}}(N_{A_{1}}=k_{1},\ldots, N_{A_{n}}=k_{n})\end{equation}
for all $A_{1},\ldots,A_{n}\in\mathcal{E}$, $n\in \N$ and $k_{1},\ldots,k_{n}\in\N_{0}\cup\{\infty\}$, then $P_{\pi_{1}}=P_{\pi_{2}}$.
\end{theorem}

We will prove Theorem \ref{sigma-theorem} in Section 4, where we study generators of $\mathcal{C}(\mathcal{S})$ under $\sigma$-finiteness. We call a cr-set $\pi$ a Poisson process if its law satisfies the following two conditions: With respect to $P_{\pi}$, the random variables $N_{A_{1}},\ldots ,N_{A_{n}}$ are independent for any disjoint $A_{1},\ldots ,A_{n}\in\mathcal{S}$ and, secondly, $N_{A}$ has a Poisson distribution for any $A\in\mathcal{S}$. At this point it is important to mention that $\delta_{0}$ and $\delta_{\infty}$ are considered to be Poisson distributions.

 The following example shows that $\sigma$-finiteness is weaker than finiteness on bounded sets. Note that finiteness on bounded sets implies $\sigma$-finiteness, when working on metric spaces.

\begin{example} {\bf (Cr-Set with Point of Accumulation)} \label{example1}Let the state space $(S,\mathcal{S})$ be the real numbers $\R$ equipped with its Borel-$\sigma$-field $\mathcal{B}(\R)$. Let $\lambda$ be the Lebesgue-measure and define $\mu(A):=\sum_{n\in\N}\lambda(A\cap (-\frac{1}{n},\frac{1}{n}))$ for $A\in\mathcal{S}$. By the existence theorem in \cite[Section 2.5]{kingman} there exists a Poisson process $\pi$ with intensity measure $\mu$. Then $\pi$ is not finite on bounded sets, but possesses a version that is $\sigma$-finite on the closed subsets of $\R$.
\end{example}

Most of our work on uniqueness in law in this paper was motivated by R\'{e}nyi's Theorem on Poisson processes, as presented in \cite[Section 3.4]{kingman}. Assuming finiteness on bounded sets, R\'{e}nyi proved in \cite{renyi} that the law of a Poisson processes $\pi$ on $\R$ is determined by the \emph{hitting probabilities} $P(\pi\cap A\neq \emptyset)$ of sets $A$ taken from a ring generating the Borel-sets of $\R$. M\"onch then showed in \cite{moench} that R\'{e}nyi's idea works in general for polish state spaces and simple point processes that are finite on bounded sets. We are going to show that M\"onch's assumptions can still be weakened. To begin with, $\sigma$-finiteness suffices again:

\begin{theorem} \label{theo-unique-ring}
Let $\Delta\in\mathcal{S}\otimes\mathcal{S}$ and let $\mathcal{R}\subseteq \mathcal{S}$ be a ring with $\sigma(\mathcal{R})=\mathcal{S}$. If $\pi_{1},\pi_{2}: \Om\to C(S)$ are two cr-sets satisfying 
   \[ P(\pi_{1}\cap A\neq \emptyset)=P(\pi_{2}\cap A\neq \emptyset)\] for all $A\in\mathcal{R}$ and if $\pi_{1}$ is $\sigma$-finite on $\mathcal{R}$, then $P_{\pi_{1}}=P_{\pi_{2}}$. \end{theorem}

Note that in Theorem \ref{theo-unique-ring} $\sigma$-finiteness is required only for one of the two cr-sets. The proof of Theorem \ref{theo-unique-ring} is provided in Section 5, where we study hitting functions of cr-sets. The \emph{hitting function} of a cr-set $\pi$ is the map $\mathcal{S}\to [0,1]$ assigning each $A\in\mathcal{S}$ its hitting probability $P(\pi\cap A\neq \emptyset)$. If any of the cr-sets involved is not $\sigma$-finite, the methods of Section 4 do not work. This is shown in Section 3, which deals with the limitations of use of generators in $(C(S),\mathcal{C}(\mathcal{S}))$. Therefore Theorem \ref{theo-unique-ring} will be proven by analysis of hitting functions.

Applying Theorem \ref{theo-unique-ring} together with the existence theorem for Poisson processes \cite[Section 2.5]{kingman}, we get the following version of R\'{e}nyi's theorem.  

\begin{corollary} {\bf (R\'{e}nyi's Theorem, Version 1)}
 Let $\Delta\in\mathcal{S}\otimes\mathcal{S}$ and let $\mathcal{R}\subseteq \mathcal{S}$ be a ring with $\sigma(\mathcal{R})=\mathcal{S}$. If a measure $\mu:\mathcal{S}\to [0,\infty]$ is $\sigma$-finite on $\mathcal{R}$ with $\mu(\{x\})=0$ for all $x\in S$, and if $\pi$ is a cr-set satisfying \[ P(\pi\cap A=\emptyset)=e^{-\mu(A)} \] for all $A\in\mathcal{R}$, then $\pi$ is a Poisson process with intensity $\mu$.
\end{corollary}

The assumption of $\sigma$-finiteness can be weakened further to constructiveness if we switch to other determining classes instead of a ring.

\begin{definition}\label{defi-constr}
Let $\tau:\Om\to C(S)$ be a map. If $\tau$ is the union of countably many finite cr-sets $\pi_{k}$, i.e. $\tau(\om)=\bigcup_{k\in\nz}\pi_{k}(\om)$ for all $\om\in\Om$, $\tau$ is called \emph{constructive}. 
\end{definition} 

Provided $S$ is a topological space, countable intersections of open sets are called {\it $G_{\delta}$-sets}. Working on a separable metric space, we get the following characterisation by hitting probabilities for the law of a constructive cr-set:

\begin{theorem} \label{theo-unique-closed}
 Let $S$ be a separable metric space and let $\mathcal{S}$ be its Borel-$\sigma$-field. For two cr-sets $\pi_{1}$ and $\pi_{2}$ the following statements hold.
 \begin{satz-liste}
 \item[a)] If $\pi_{1}$ and $\pi_{2}$ are both constructive and satisfy
  \begin{equation}\label{eq-neu1}P(\pi_{1}\cap F\neq \emptyset)=P(\pi_{2}\cap F\neq \emptyset)\;\;\;\;\mbox{for all closed}\;\; F\subseteq S, \end{equation} then $P_{\pi_{1}}=P_{\pi_{2}}$.
  \item[b)] Suppose only that $\pi_{2}$ is constructive, then $P_{\pi_{1}}=P_{\pi_{2}}$ follows from (\ref{eq-neu1}), under the additional condition
  \begin{equation}\label{eq-tuc2}P(\pi_{2}\cap A\neq \emptyset)=0\;\:\mbox{implies}\;\:  P(\pi_{1}\cap A\neq \emptyset)=0\;\:\mbox{for all}\;\:A\in\mathcal{S}.
  \end{equation} 
  \item[c)] Suppose only that $\pi_{2}$ is constructive, then
  \[ P(\pi_{1}\cap A\neq \emptyset)=P(\pi_{2}\cap A\neq \emptyset)\] for all $G_{\delta}$-sets $A\subseteq S$ yields $P_{\pi_{1}}=P_{\pi_{2}}$.
 \end{satz-liste}
    
\end{theorem}

The proof of Theorem \ref{theo-unique-closed} is provided in Section 5 as well. Similarly as above, combining the uniqueness Theorem \ref{theo-unique-closed} and the existence theorem, we obtain another version of R\'{e}nyi's theorem. This is possible, since the existence theorem provides always a constructive Poisson process.

\begin{corollary} {\bf (R\'{e}nyi's Theorem, Version 2)}
Let $S$ be a separable metric space, $\mathcal{S}$ its Borel-$\sigma$-field and let $\pi$ be a constructive cr-set. If for all $n\in\N$ $\mu_{n}:\mathcal{S}\to [0,\infty]$ is a finite measure with $\mu_{n}(\{x\})=0$ for all $x\in S$, and if \[ P(\pi\cap F=\emptyset)=e^{-\sum_{n\in\N}\mu_{n}(F)} \] for all closed sets $F\subseteq S$,
then $\pi$ is a Poisson process with intensity $\sum_{n\in\N}\mu_{n}$.
\end{corollary}

Besides uniqueness in law of cr-sets, we are also going to study Definition \ref{defi-constr} in this paper. This is done in Section 6. We will prove that our notion of constructiveness is equivalent to  Kendall's ``constructive countability'' \cite{kendall}: For a cr-set $\pi:\Om\to C(S)$ and $E\in\mathcal{F}$ define $1_{E}\,\pi:\Om\to C(S)$ by
\[ 1_{E}\,\pi(\om):= \left\{ \begin{array}{cc} \pi(\om) & \mbox{for}\:\om\in E,  \\ \emptyset & 
\mbox{otherwise.} \end{array} \right.\]
Like $\pi$, the map $1_{E}\,\pi$ is $\mathcal{F}$-$\mathcal{C}(\mathcal{S})$ measurable. Let $X_{n}: (\Om,\mathcal{F})\to (S,\mathcal{S})$, $n\in\N$, be a sequence of random variables. Since $\om\mapsto \{X_{n}(\om)\}$ are finite cr-sets, so are $1_{E}\,\{X_{n}\}$ and it follows that the map $1_{E}\,\{X_{1},X_{2},\ldots\}$ is constructive. If $\Delta\in\mathcal{S}\otimes\mathcal{S}$, then also the converse is true and any constructive map has this special form, that Kendall calls ``constructive countability''. This is shown in the next theorem, which can be regarded as a measurable selection theorem for constructive maps.

\begin{theorem} \label{theo-constructive1}
 Let $\Delta\in\mathcal{S}\otimes\mathcal{S}$ and let $\tau :\Om\to C(S)$ be constructive. Then there exist random variables $X_{n}:(\Om,\mathcal{F})\to (S,\mathcal{S})$, $n\in\N$, with $\tau=1_{\{\tau\neq\emptyset\}}\,\{X_{1},X_{2},\ldots \}$.
\end{theorem}

Any $\sigma$-finite cr-set is constructive. The converse is not true, as the following example shows.

\begin{example} \label{example2} {\bf (Cr-Set with Random Point of Accumulation)} Let $\pi$ be a $\sigma$-finite version of the Poisson process in Example \ref{example1}. Then, by Theorem \ref{theo-constructive1}, there exist random variables $X_{1},X_{2}\ldots :(\Om,\mathcal{F})\to (\R,\mathcal{B}(\R))$ with $\pi=1_{\{\pi\neq\emptyset\}}\,\{X_{1},X_{2},\ldots \}$.\\ We are going to randomly shift $\pi$ on $\R$. For this purpose, let $Z:(\Om,\mathcal{F})\to (\R,\mathcal{B}(\R))$ be a random variable that is independent of $X_{1},X_{2},\ldots$ and let the law of $Z$ be equivalent to $\lambda$. If we define $\tau:=1_{\{\pi\neq\emptyset\}}\,\{X_{1}+Z,X_{2}+Z,\ldots \}$, we obtain $0<P(N_{A}(\tau)=\infty)<1$ for all bounded Borel-sets $A\subset \R$ with $\lambda(A)>0$. The cr-set $\tau$ is constructive and not $\sigma$-finite.
 
\end{example}

We will give an explanation of Example \ref{example2} at the end of Section 6. We believe that constructive cr-sets with random points of accumulation might be an interesting class of cr-sets to study.

 A constructive cr-set can be decomposed into a $\sigma$-finite part and a part that is nowhere $\sigma$-finite. This is the subject of the next theorem.
Note that $\pi\cap A$ is a cr-set for any cr-set $\pi:\Om\to C(S)$ and for any $A\in\mathcal{S}$.

\begin{theorem}\label{theo-decomp}
 Let $\Delta\in\mathcal{S}\otimes\mathcal{S}$ and let $\pi$ be a constructive cr-set. Then there exists $F\in\mathcal{S}$ such that
\begin{enumerate}
 \item[(i)] $\pi\cap F$ possesses a version that is $\sigma$-finite on $\mathcal{S}$,
 \item[(ii)] for all $A\in\mathcal{S}$ holds $P((\pi\cap F^{c})\cap A\neq \emptyset)=0$ or $P(|(\pi\cap F^{c})\cap A|=\infty)>0$.
\end{enumerate}
If two sets $F_{1},F_{2}\in\mathcal{S}$ satisfy (i) and (ii), then $P(\pi\cap (F_{1}\vartriangle F_{2})\neq \emptyset)=0$.
\end{theorem}

A cr-set satisfying the first condition in the definition of Poisson processes is said to have {\it independent increments}. It is known that for $\sigma$-finite cr-sets, hitting singletons with probability 0, the Poisson distribution is a necessary consequence of independent increments -- see e.g. \cite[Theorem 12.10]{kallenberg}. We are going to show that this statement also holds for constructive cr-sets.

\begin{theorem} \label{theo-inde-incre}Let $\Delta\in\mathcal{S}\otimes\mathcal{S}$ and let $\pi$ be a constructive cr-set with independent increments. If $P(\pi\cap \{x\}\neq \emptyset)=0$ for all $x\in S$, then $\pi$ is a Poisson process.
\end{theorem}

Theorem \ref{theo-constructive1}, Theorem \ref{theo-decomp} and Theorem \ref{theo-inde-incre} are all proven in Section 6.


\subsection{Spaces with Measurable Diagonal and Dissecting Systems} \label{tools}
\setcounter{theorem}{0}

In this section we discuss some techniques that we are going to use throughout this paper.

For $x,y\in S$ a collection $\mathcal{E}$ of subsets of $S$ is said to {\it separate} $x$ and $y$ if there exists some $A\in \mathcal{E}$ with $1_{A}(x)\neq 1_{A}(y)$. We say $\mathcal{E}$ {\it separates points} of a subset $A\subseteq S$ if for any two distinct points $x,y\in A$ the collection $\mathcal{E}$ separates $x$ and $y$. The next proposition is part of \cite[Theorem 1]{dravecky} and \cite[Remark 3]{dravecky}.

\begin{prop} \label{prop2-1}
Let $(S,\mathcal{S})$ be any measurable space. The following statements are equivalent.
\begin{satz-liste}
\item[a)] $\Delta\in\mathcal{S}\otimes\mathcal{S}$.
\item[b)] There is a countable collection $\mathcal{E}\subseteq \mathcal{S}$ separating points of $S$.
\item[c)] There is a metric $d:S\times S\to [0,\infty)$, such that $(S,d)$ is separable and $\mathcal{S}$ contains the Borel-$\sigma$-field of $(S,d)$.
\end{satz-liste}
\end{prop}

In particular, $\Delta\in\mathcal{S}\otimes\mathcal{S}$ implies $\{x\}\in\mathcal{S}$ for all $x\in S$ and a separable metric space $S$ with Borel-$\sigma$-field $\mathcal{S}$ always has a measurable diagonal. Part b) of Proposition \ref{prop2-1} will be the most useful characterisation of $\Delta\in\mathcal{S}\otimes\mathcal{S}$ for us. The fact that $\mathcal{S}$ is a $\sigma$-field is often not as important as the existence of a countable subsystem separating points of $S$. Note that $\Delta\in\mathcal{S}\otimes\mathcal{S}$ is a requirement on the richness of $\mathcal{S}$ -- provided the set $S$ permits the existence of a $\sigma$-field with measurable diagonal at all.

Another very important concept for us are dissecting systems. The following definition is due to Leadbetter \cite{leadbetter}: 

\smallskip
\begin{definition} \label{defi-dissys}
Let $A\subseteq S$ and for all $n\in\N$ let $A_{n,k}\subseteq S$, $k\in I_{n}\subseteq\N$, be a collection of subsets. Then $(\{A_{n,k}\,|\,k\in I_{n}\})_{n\in\N}$ is called a \emph{dissecting system} for $A$, if the following conditions hold:
\begin{enumerate}
 \item[(i)] For all $n\in\N$ the sets in $\{ A_{n,k}\,|\, k\in I_{n}\}$ are pairwise disjoint and satisfy $\bigcup_{k\in I_{n}}A_{n,k}\subseteq \bigcup_{k\in I_{n+1}}A_{n+1,k}\subseteq A$, as well as $\bigcup_{n\in\N,k\in I_{n}}A_{n,k}=A$.
\item[(ii)] For any two distinct points $x,y\in A$ there exists $n(x,y)\in \N$ such that $\{ A_{n,k}\,|\, k\in I_{n}\}$ separates $x$ and $y$ for all $n\geq n(x,y)$.
\end{enumerate}
\end{definition}

In fact, measurability of the diagonal is equivalent to the existence of a measurable dissecting system for $S$, as essentially shown in \cite[Theorem 1]{dravecky}. The following lemma will help us to make use of dissecting systems. It is also due to Leadbetter \cite[Lemma 2.1]{leadbetter}. Note that in \cite[Lemma 2.1]{leadbetter} convergence is stated almost surely, although it holds pointwise as well.

\begin{lemma} \label{lemma-leadbetter}Let $(\{A_{n,k}\,|\, k\in I_{n}\})_{n\in\N}$ be a dissecting system for $A\subseteq S$. Then for all $M\in C(S)$ holds \begin{equation}\label{eq-leadbetter} N_{A}(M)=\lim_{n\to\infty}\sum_{k\in I_{n}}1_{\{N_{A_{n,k}}\neq 0\}}(M).\end{equation} 

\end{lemma}

\begin{definition}\label{defi-selfdis}
A family $\emptyset\neq\mathcal{E}\subseteq \mathcal{P}(S)$ is called \emph{self-dissecting} if for all $A\in\mathcal{E}$ there exists a dissecting system $(\{A_{n,k}\,|\, k\in I_{n}\})_{n\in\N}$ with $\{A_{n,k}\,|\,k\in I_{n}\}\subseteq \mathcal{E}$ for all $n\in\N$.

\end{definition}

Lemma \ref{lemma-leadbetter} and Definition \ref{defi-selfdis} are going to be important in the next section, when we study ``hit or miss'' $\sigma$-fields.

Next, we remind of two lemmata from measure theory that we are going to use a couple of times in this paper. The first one can be found basically in \cite[\S 5 Theorem D.]{Halmos}. Unfortunately, we don't have any reference for the second lemma. Let $\mathfrak{X}\neq\emptyset$ be any set.

\begin{lemma} \label{maß-lemma1}
Let $\emptyset\neq\mathcal{E}\subseteq\mathcal{P}(\mathfrak{X})$ be a collection of subsets. For each $A\in\sigma(\mathcal{E})$ there exist countable many sets $A_{n}\in\mathcal{E}, n\in\N$, such that $A\in\sigma(\{A_{n}\,|\,n\in\N\})$.
\end{lemma}

\begin{lemma} \label{maß-lemma2}
Let $\emptyset\neq\mathcal{E}\subseteq\mathcal{P}(\mathfrak{X})$ be a collection of subsets and let $x,y\in\mathfrak{X}$. If $1_{A}(x)=1_{A}(y)$ holds for all $A\in \mathcal{E}$, then it holds already for all $A\in\sigma(\mathcal{E})$.
\end{lemma}

\begin{proof}[\normalfont\bfseries Proof.] The collection $\{ A\in\sigma(\mathcal{E})\,|\,1_{A}(x)=1_{A}(y)\}$ is a $\sigma$-field.
\end{proof}

Finally in this section, we state some propositions that provide us with self-dissecting set-systems.

\begin{prop} \label{prop-selfdis1}
Let $\mathcal{H}\subseteq\mathcal{P}(S)$ be a semiring. If there is a countable collection $\mathcal{E}\subseteq \mathcal{H}$ separating points of $S$, then $\mathcal{H}$ is self-dissecting.
 
\end{prop}

\begin{proof}[\normalfont\bfseries Proof.]
  Let $A\in\mathcal{H}$. We are going to construct inductively a nested sequence of partitions for $A$ -- an idea that can be found in the proof of \cite[Theorem 1]{dravecky} or in the proof of \cite[Proposition A2.1.V]{daley}. For this purpose, let $E_{1},E_{2},\ldots$ be the elements of $\mathcal{E}$. Since $\mathcal{H}$ is a semiring, we have $A\cap E_{1}\in\mathcal{H}$ and there exist $l\in\N$ and pairwise disjoint $H_{1},\ldots,H_{l}\in\mathcal{H}$ with $A\setminus E_{1}=H_{1}\cup\cdots\cup H_{l}$. Define the first partition $\{A_{1,k}\,|\,k\in I_{1}\}$ to be $\{A\cap E_{1},H_{1},\ldots, H_{l}\}$.\\ Given $\{A_{n,k}\,|\,k\in I_{n}\}\subseteq \mathcal{H}$ for $n\in\N$, we proceed with each $A_{n,k}$ in the same fashion: For all $k\in I_{n}$ we have $A_{n,k}\cap E_{n+1}\in\mathcal{H}$ and there exist $l_{k}\in\N$ and pairwise disjoint $H_{k,1},\ldots,H_{k,l_{k}}\in\mathcal{H}$ with $A_{n,k}\setminus E_{n+1}=H_{k,1}\cup\cdots\cup H_{k,l_{k}}$. Define $\{A_{n+1,k}\,|\,k\in I_{n+1}\}$ to be $\bigcup_{k\in I_{n}}\{A_{n,k}\cap E_{n+1},H_{k,1},\ldots, H_{k,l_{k}}\}$.\\
It remains to check that $(\{A_{n,k}\,|\, k\in I_{n}\})_{n\in \N}$ is a dissecting system for $A$.
Since the sequence of partitions is nested, if $\{A_{n,k}\,|\,k\in I_{n}\}$ separates two points in $A$, then so does $\{A_{n+1,k}\,|\,k\in I_{n+1}\}$. Therefore and because $\mathcal{E}$ is separating points of $A$, we conclude that condition (ii) in Definition \ref{defi-dissys} is satisfied. Condition (i) is easy to check in this case.
\end{proof}

\begin{prop}\label{prop-selfdis2} Let $(S,\mathcal{S})$ be a measurable space with $\Delta\in\mathcal{S}\otimes\mathcal{S}$ and let $\mathcal{H}\subseteq\mathcal{S}$ be a semiring generating $\mathcal{S}$. Then $\mathcal{H}$ is self-dissecting.

\end{prop}

\begin{proof}[\normalfont\bfseries Proof.]
By Proposition \ref{prop2-1}, there exists a family of sets $\{A_{m}\,|\,m\in\N\}\subseteq\mathcal{S}$ separating points of $S$. Since $\mathcal{S}=\sigma(\mathcal{H})$, we can apply Lemma \ref{maß-lemma1} and receive a countable $\mathcal{E}_{m}\subseteq \mathcal{H}$ with $A_{m}\in\sigma(\mathcal{E}_{m})$ for all $m\in\N$.\\
Define $\mathcal{E}:=\bigcup_{m\in\N}\mathcal{E}_{m}$. Then $\mathcal{E}$ is countable and satisfies $\mathcal{E}\subseteq\mathcal{H}$. Since $\{A_{m}\,|\,m\in\N\}\subseteq \sigma(\mathcal{E})$ and since $\{A_{m}\,|\,m\in\N\}$ is separating points of $S$, we conclude by Lemma \ref{maß-lemma2} and by a contradiction argument that $\mathcal{E}$ is separating points of $S$ as well. Therefore Proposition \ref{prop-selfdis2} follows from Proposition \ref{prop-selfdis1}.\end{proof}

\begin{prop} \label{prop-selfdis3}
 Let $\mathcal{E}\subseteq\mathcal{P}(S)$ be a collection of subsets. If $\mathcal{E}$ is $\cap$-stable and contains a dissecting system for $S$, then $\mathcal{E}$ is self-dissecting.
\end{prop}
Proposition \ref{prop-selfdis3} is easy to check. We will not give a proof.


\subsection{Generators of Counting $\sigma$-Fields in $C(S)$ and ``Hit or Miss'' $\sigma$-Fields} \label{generators1}
\setcounter{theorem}{0}

Theorems on uniqueness in law for cr-sets or simple point processes can often be explained by finding a $\cap$-stable generator for the $\sigma$-field in the ``the right'' measurable space and then applying the uniqueness theorem of measure theory \cite[Lemma 1.17]{kallenberg}. For example, this method is used in the proof of Theorem \ref{sigma-theorem} in Section \ref{generators2}. However, in this section we will try to find limitations of that line of argumentation, when working in $C(S)$. Therefore we are going to study generators of $\sigma$-fields in $C(S)$ that are of the form \[ \mathcal{C}(\mathcal{T}):=\sigma(N_{A}\,|\,A\in\mathcal{T}).\] In contrast to Section 1, we allow $\emptyset\neq\mathcal{T}\subseteq\mathcal{P}(S)$ here to be an arbitrary collection of ``test sets'', on that we would like to be able to count points. We call $\mathcal{C}(\mathcal{T})$ the \emph{counting $\sigma$-field} of $\mathcal{T}$.

In order to study $\mathcal{C}(\mathcal{T})$ and also in general, the concept of ``hit or miss'' $\sigma$-fields  plays an important role. We define the \emph{``hit or miss'' $\sigma$-field} of $\mathcal{T}$ in $C(S)$ to be 
\[ \mathfrak{H}(\mathcal{T}):=\sigma(\{N_{A}\neq 0\}\,|\,A\in\mathcal{T}).\] 
According to Kendall \cite{kendall} ``hit or miss'' $\sigma$-fields were introduced to stochastic geometry by Matheron \cite{matheron}.

\begin{theorem} \label{theo-hit}
 If $\mathcal{E}\subseteq \mathcal{P}(S)$ is a self-dissecting collection of sets, then
\[\mathfrak{H}(\mathcal{E})= \mathcal{C}(\mathcal{E}).\]
\end{theorem}

\begin{proof}[\normalfont\bf Proof.] 
 Obviously we have $\mathfrak{H}(\mathcal{E})\subseteq \mathcal{C}(\mathcal{E})$. For the other inclusion, it suffices to show that $N_{A}$ is measurable with respect to $\mathfrak{H}(\mathcal{E})$ for all $A\in\mathcal{E}$. Let $A\in\mathcal{E}$ and let $(\{A_{n,k}\,|\, k\in I_{n}\})_{n\in\N}$ be a dissecting system for $A$ with $\{A_{n,k}\,|\,k\in I_{n}\}\subseteq \mathcal{E}$ for all $n\in\N$. Then the maps $1_{\{N_{A_{n,k}}\neq 0\}}$ are $\mathfrak{H}(\mathcal{E})$-measurable for all $n\in\N$, $k\in I_{n}$, and we conclude by equation (\ref{eq-leadbetter}) of Lemma \ref{lemma-leadbetter} that so is $N_{A}$.
\end{proof}

\begin{corollary} \label{prop2-3}
If $\Delta\in\mathcal{S}\otimes\mathcal{S}$, then $\{\{N_{A}= 0\}\,|\,A\in \mathcal{S}\}$ is  $\cap$-stable and generates $\mathcal{C}(\mathcal{S})$.
\end{corollary}

\begin{proof}[\normalfont\bf Proof.] Since $\{N_{A}=0\}\cap \{N_{B}=0\}=\{N_{A\cup B}=0\}$ for all $A,B\in\mathcal{S}$, the system is $\cap$-stable. Measurability of the diagonal implies that $\mathcal{S}$ is self-dissecting. Therefore we can apply Theorem \ref{theo-hit}. \end{proof}

\begin{theorem} \label{prop2-6}
 Let $\emptyset\neq\mathcal{E}\subseteq\mathcal{P}(S)$. For all $A \subseteq S$, the set $\{N_{A}\neq 0\}$ is an element of $\mathfrak{H}(\mathcal{E})$
if and only if $A$ is the union of countably many sets in $\mathcal{E}$.
\end{theorem}

\begin{proof}[\normalfont\bfseries Proof]
 If $A=\bigcup_{n\in\N}E_{n}$ with $E_{n}\in \mathcal{E}$, then $\{N_{A}\neq 0\}=\bigcup_{n\in\N}\{N_{E_{n}}\neq 0\}$ belongs to $\mathfrak{H}(\mathcal{E})$. On the other hand, let $\{N_{A}\neq 0\}$ be an element of $\mathfrak{H}(\mathcal{E})$. Then by Lemma \ref{maß-lemma1} there exist $E_{n}\in\mathcal{E}$ with  $\{N_{A}\neq 0\}\in\sigma( \{N_{E_{n}}\neq 0\}\,|\,n\in\N )$. Now for any $M_{1},M_{2}\in C(S)$ Lemma \ref{maß-lemma2} yields \[ \big(\forall n\in\N:1_{\{N_{E_{n}}\neq 0\}}(M_{1})=1_{\{N_{E_{n}}\neq 0\}}(M_{2})\big)\Rightarrow 1_{\{N_{A}\neq 0\}}(M_{1})=1_{\{N_{A}\neq 0\}}(M_{2}).\]  
Formulating this statement without indicator variables, we get 
\begin{equation}\label{eq2} \hspace{-1mm}\left(\forall n\in\N:M_{1}\cap E_{n}\neq \emptyset \Leftrightarrow M_{2}\cap E_{n}\neq \emptyset\right) \Rightarrow \left( M_{1}\cap A\neq \emptyset \Leftrightarrow M_{2}\cap A\neq\emptyset \right)\end{equation} for all 
$M_{1},M_{2}\in C(S)$.

Finally, let $\mathcal{I}:=\{E_{n}\,|\,E_{n}\subseteq A,\,n\in\N\}$. We show $\bigcup_{E\in \mathcal{I}}E=A$ by contradiction and therefore suppose $\bigcup_{E\in \mathcal{I}}E\subsetneq A$. Then there exists $x\in A$ with $x\notin \bigcup_{E\in \mathcal{I}}E$.\\
In the case of $x\notin \bigcup_{n\in\N}E_{n}$, set $M_{1}:=\{x\}$ and $M_{2}:=\emptyset$. Then $M_{1}$ and $M_{2}$ contradict implication (\ref{eq2}).\\
If $x$ is an element of $\bigcup_{n\in\N}E_{n}$, let $\{i_{1},i_{2},\ldots \}:=\{n\in\N\,|\,x\in E_{n}\}$. Since $x\notin \bigcup_{E\in\mathcal{I}}E$, for all $i_{k}$ there exists $y_{k}\in A^{c}\cap E_{i_{k}}$. Set $M_{1}:=\{y_{1},y_{2},\ldots \}$ and $M_{2}:=\{x\}\cup M_{1}$. Again $M_{1}$ and $M_{2}$ contradict implication (\ref{eq2}).\end{proof}

\begin{lemma}\label{lemma-generator} Let $\emptyset\neq\mathcal{E}, \mathcal{T}\subseteq\mathcal{P}(S)$ be collections of ``test sets'' satisfying $\mathcal{C}(\mathcal{T})\subseteq \mathcal{C}(\mathcal{E})$. If $\mathcal{T}$ contains a countable subset separating points of $S$, then so does $\mathcal{E}$.
\end{lemma}

\begin{proof}[\normalfont\bf Proof.]
Let $\{A_{n}\,|\,n\in\N\}\subseteq\mathcal{T}$ be separating points of $S$.  For all $n\in\N$ we have $\{N_{A_{n}}\neq 0\}\in \mathcal{C}(\mathcal{T})\subseteq \mathcal{C}(\mathcal{E})$. Therefore, by Lemma \ref{maß-lemma1}, there exist $E_{n,m}\in\mathcal{E}$ and $K_{n,m}\subseteq \N_{0}\cup\{\infty\}$ with $\{N_{A_{n}}\neq 0\}\in\sigma(\{N_{E_{n,m}}\in K_{n,m}\}\,|\,m\in\N)$.\\
Then $\{E_{n,m}\,|\,n,m\in\N\}\subseteq \mathcal{E}$ separates points of $S$, as can be seen by a contradiction argument.
Suppose there exist $x\neq y\in S$ with $1_{E_{n,m}}(x)= 1_{E_{n,m}}(y)$ for all $n,m\in\N$. We conclude $N_{E_{n,m}}(\{x\})=N_{E_{n,m}}(\{y\})$ and, consequently, $1_{\{N_{E_{n,m}}\in K_{n,m}\}}(\{x\})=1_{\{N_{E_{n,m}}\in K_{n,m}\}}(\{y\})$ for all $n,m\in\N$. Hence Lemma \ref{maß-lemma2} implies $1_{\{N_{A_{n}}\neq 0\}}(\{x\})=1_{\{N_{A_{n}}\neq 0\}}(\{y\})$ for all $n\in\N$, which is equivalent to $1_{A_{n}}(x)=1_{A_{n}}(y)$ for all $n\in\N$.\end{proof}

\begin{theorem}\label{theo-generator1} Let $\mathcal{T}\subseteq \mathcal{P}(S)$ be a family of ``test sets'' containing a countable subset separating points of $S$ and let $\mathcal{H}\subseteq\mathcal{P}(S)$ be a semiring. Then $\mathcal{C}(\mathcal{T})\subseteq\mathcal{C}(\mathcal{H})$ if and only if all elements in $\mathcal{T}$ are countable unions of $\mathcal{H}$-sets.

\end{theorem}

\begin{proof}[\normalfont\bf Proof.]
Suppose $\mathcal{C}(\mathcal{T})\subseteq\mathcal{C}(\mathcal{H})$. Applying Lemma \ref{lemma-generator} and then Proposition \ref{prop-selfdis1}, we conclude that $\mathcal{H}$ is self-dissecting. Therefore Theorem \ref{theo-hit} yields $\mathcal{C}(\mathcal{H})=\mathfrak{H}(\mathcal{H})$. Since $\mathcal{C}(\mathcal{T})\subseteq\mathfrak{H}(\mathcal{H})$, it follows from Theorem \ref{prop2-6} that all sets in $\mathcal{T}$ are countable unions of $\mathcal{H}$-sets.\\
On the other hand, suppose that all $\mathcal{T}$-sets are countable unions of $\mathcal{H}$-sets and let $A\in\mathcal{T}$. We need to show that $N_{A}$ is measurable with respect to $\mathcal{C}(\mathcal{H})$. Since $\mathcal{H}$ is a semiring, there exist pairwise disjoint $H_{1},H_{2}\ldots\in\mathcal{H}$ with $A=\bigcup_{n\in\N}H_{n}$. Hence we deduce from $N_{A}=\sum_{n=1}^{\infty}N_{H_{n}}$ that $N_{A}$ is $\mathcal{C}(\mathcal{H})$-measurable. 
\end{proof}

\begin{corollary} \label{coro-generator}Let $\emptyset\neq\mathcal{T}\subseteq \mathcal{P}(S)$ be a family of ``test sets'' containing a countable subset separating points of $S$. If $\mathcal{C}(\mathcal{T})\subseteq \mathcal{C}(\mathcal{E})$  for $\emptyset\neq\mathcal{E}\subseteq \mathcal{P}(S)$, then all $A\in\mathcal{T}$ are of the form
\[ A=\bigcup_{n\in\N} E^{*}_{n,1}\cap\cdots\cap E^{*}_{n,k_{n}}, \]
where $E_{n,j}\in\mathcal{E}$ and $E_{n,j}^{*}$ is either $E_{n,j}$ or its complement.

\end{corollary}

\begin{proof}[\normalfont\bf Proof.]
For all $\emptyset\neq\mathcal{E}\subseteq\mathcal{P}(S)$ the system \[\mathcal{E}^{*}:=\{ E_{1}^{*}\cap\cdots\cap E_{n}^{*}\,|\,\:n\in\N,\: E_{i}^{*}=E_{i}\;\mbox{or}\;E_{i}^{*}=E_{i}^{c}\;\mbox{for}\;E_{i}\in\mathcal{E}\}\] is a semiring containing $\mathcal{E}$. If $\mathcal{C}(\mathcal{T})\subseteq\mathcal{C}(\mathcal{E})$, we conclude $\mathcal{C}(\mathcal{T})\subseteq\mathcal{C}(\mathcal{E}^{*})$ and apply Theorem \ref{theo-generator1}.
\end{proof}

Corollary \ref{coro-generator} shows that generators for $\mathcal{C}(\mathcal{S})$ of the form $\{N_{A}\,|\,A\in\mathcal{E}\}$ cannot be substantially simpler than $\{N_{A}\,|\,A\in\mathcal{S}\}$ itself. However, for a constructive cr-sets it is possible to find a substantially simpler class of sets in $\mathcal{C}(\mathcal{S})$ determining its law.
For example, let $\mathcal{H}:=\{F\cap G\,|\,F,G\subseteq\R,\,F\;\mbox{closed},\,G\;\mbox{open}\}$ and let $\mathcal{S}$ be the Borel-$\sigma$-field in $\R$. The collection $\mathcal{H}\subseteq \mathcal{S}$ is a semiring and each element of $\mathcal{H}$ is the union of countably many closed sets. If $\mathcal{C}(\mathcal{H})=\mathcal{C}(\mathcal{S})$, then by Theorem \ref{theo-generator1} every set in $\mathcal{S}$ was the countable union of closed sets. Since this is not the case, we conclude $\mathcal{C}(\mathcal{H})\subsetneq \mathcal{C}(\mathcal{S})$, which also implies $\mathfrak{H}(\mathcal{H})\subsetneq \mathcal{C}(\mathcal{S})$.
Nevertheless, the law of a constructive cr-set is determined by its values on $\{\{N_{F}\neq 0\}\,|\,F\subseteq \R\:\mbox{closed}\}$, as shown in Theorem \ref{theo-unique-closed}.


\subsection{Generators of $\mathcal{C}(\mathcal{S})$ under $\sigma$-Finiteness} \label{generators2}
\setcounter{theorem}{0}

In this section we study generators of the counting $\sigma$-field $\mathcal{C}(\mathcal{S})$, when working with a finiteness condition on the space $C(S)$. In contrast to the last section, generators can be quite simple here.

Let $\sf F\subseteq\mathcal{P}(S)$ be a collection of sets. Instead of $C(S)$ we are going to work in the following with the subspace \[ C^{\sf F}(S):=\{M\in C(S)\,|\,\mbox{for all}\:A\in{\sf F}\:\mbox{holds}\:|M\cap A|<\infty \} \]
of countable subsets of $S$ which are ``finite on {\sf F}''.
If $S$ is a metric space, the usual choice is ${\sf F}=\{ A\subseteq S\,|\,A\:\mbox{is bounded}\}$. However, in our context $C^{\sf F}(S)$ should be considered as ``the right'' space for $\sigma$-finite cr-sets. Therefore it will be useful later to let $\sf F$ be a countable covering of $S$, i.e. ${\sf F}=\{A_{n}\,|\,n\in\N\}$ with $S=\bigcup_{n\in\N}A_{n}$.

In order to introduce counting $\sigma$-fields in $C^{\sf F}(S)$, let $N_{A}^{\sf F}: C^{\sf F}(S)\to \N_{0}\cup\{\infty\}$, $M\mapsto |M\cap A|$ be the restriction of $N_{A}$ to $C^{\sf F}(S)$ for all $A\subseteq S$. We then define \[ \mathcal{C}^{\sf F}(\mathcal{T}):=\sigma(N_{A}^{\sf F}\,|\,A\in\mathcal{T})\] for any non-empty collection of ``test sets'' $\mathcal{T}\subseteq \mathcal{P}(S)$.

 Let $\mathfrak{X}\neq \emptyset$ be any set. For any collection $\mathcal{E}\subseteq \mathcal{P}(\mathfrak{X})$ and $A\in\mathcal{P}(\mathfrak{X})$ define $\mathcal{E}\cap A:=\{E\cap A\,|\, E\in \mathcal{E}\}$. If $\mathcal{E}$ is a $\sigma$-field on $\mathfrak{X}$, then $\mathcal{E}\cap A$ is a $\sigma$-field on $A$, which is called the \emph{trace $\sigma$-field} of $\mathcal{E}$ on $A$. Remember that $\sigma(\mathcal{E})\cap A$ is generated by $\mathcal{E}\cap A$ for any non-empty $\mathcal{E}\subseteq \mathcal{P}(\mathfrak{X})$ and $A\in\mathcal{P}(\mathfrak{X})$.

 This last statement implies in our context that $\mathcal{C}^{\sf F}(\mathcal{T})$ is the trace $\sigma$-field of $\mathcal{C}(\mathcal{T})$ on $C^{\sf F}(S)$, i.e. $\mathcal{C}^{\sf F}(\mathcal{T})= \mathcal{C}(\mathcal{T})\cap C^{\sf F}(S)$. Note that a map $\pi:\Om\to C^{\sf F}(S)$ can be considered having $C(S)$ as target set. Then $\pi$ is $\mathcal{F}$-$\mathcal{C}^{\sf F}(\mathcal{T})$  measurable if and only if it is $\mathcal{F}$-$\mathcal{C}(\mathcal{T})$  measurable.

The following theorem is a generalisation of \cite[Lemma 1.4]{kallenberg2}, which has a similar proof. Note that Theorem \ref{theo-generator2} and its proof work for random measures in general as well.

\begin{theorem} \label{theo-generator2}
Let $\mathcal{E}$ be a $\cap$-stable collection generating $\mathcal{S}$. If there exists $\{E_{n}\,|\,n\in\N\}\subseteq \mathcal{E}\cap {\sf F}$ with $S=\bigcup_{n\in\N}E_{n}$, then $\{N^{\sf F}_{E}\,|\,E\in\mathcal{E}\}$ generates $\mathcal{C}^{\sf F}(\mathcal{S})$.
\end{theorem}

\begin{proof}[\normalfont\bfseries Proof.]
 The relation $\mathcal{C}^{\sf F}(\mathcal{E})\subseteq \mathcal{C}^{\sf F}(\mathcal{S})$ is trivial. In order to prove the other inclusion, we first establish \begin{equation}\mathcal{C}^{\sf F}(\mathcal{S}\cap E_{n})\subseteq \mathcal{C}^{\sf F}(\mathcal{E})\;\;\mbox{for all}\;n\in\N.\label{eq-generator2}\end{equation} 
For this purpose, define $\mathcal{D}^{(n)}:=\{A\in\mathcal{S}\,|\,N^{\sf F}_{A\cap E_{n}}\:\mbox{is}\:\mathcal{C}^{\sf F}(\mathcal{E})\mbox{-measurable}\}$ and show that $\mathcal{D}^{(n)}$ is a $\lambda$-system for all $n\in\N$. For the definition of $\lambda$-systems, see \cite[Chapter 1]{kallenberg}. The condition $S\in\mathcal{D}^{(n)}$ is trivial. If  $A_{m}\subseteq S$ with $A_{m}\subseteq A_{m+1}$ for all $m\in\N$ and $A=\bigcup_{m\in\N}A_{m}$, we have 
\[ N^{\sf F}_{A}(M)=\lim_{m\to\infty}N^{\sf F}_{A_{m}}(M)\;\;\;\mbox{for all}\;M\in C^{\sf F}(S). \]
From this we conclude that $\mathcal{D}^{(n)}$ is closed under formation of increasing limits. So far things would have worked in $C(S)$.  Next, we are going to use the finiteness condition. Since $E_{n}\in  {\sf F}$, we have $N^{\sf F}_{A\cap E_{n}}(M)<\infty$ for all $A\subseteq S$ and $M\in C^{\sf F}(S)$. Therefore we get
\[ N^{\sf F}_{(B\setminus A)\cap E_{n}}(M)=N^{\sf F}_{(B\cap E_{n})\setminus(A\cap E_{n})}(M)=N^{\sf F}_{B\cap E_{n}}(M)-N^{\sf F}_{A\cap E_{n}}(M)\] for all $A\subseteq B\subseteq S$ and $M\in C^{\sf F}(S)$. This implies that $\mathcal{D}^{(n)}$ is closed under proper differences, and we have shown that $\mathcal{D}^{(n)}$ is a $\lambda$-system.
Furthermore, since $\mathcal{E}$ is $\cap$-stable, it follows that $\mathcal{E}\subseteq \mathcal{D}^{(n)}$. Hence we conclude $\mathcal{S}=\sigma(\mathcal{E})\subseteq \mathcal{D}^{(n)}$ by \cite[Theorem 1.1]{kallenberg}. Now, since $N^{\sf F}_{A\cap E_{n}}$ is $\mathcal{C}^{\sf F}(\mathcal{E})$-measurable for all $A\in\mathcal{S}$, we have established (\ref{eq-generator2}).

Finally, we show that $N^{\sf F}_{A}$ is $\mathcal{C}^{\sf F}(\mathcal{E})$-measurable for all $A\in\mathcal{S}$, which yields 
$\mathcal{C}^{\sf F}(\mathcal{S})\subseteq \mathcal{C}^{\sf F}(\mathcal{E})$. Let $A\in\mathcal{S}$ and define $E_{0}:=\emptyset$. Since $S=\bigcup_{n\in\N}E_{n}$, the set $A$ is the union of the pairwise disjoint sets $(A\cap E_{n})\setminus (A\cap (E_{0}\cup\cdots\cup E_{n-1}))$, $n\in\N$. Hence we conclude
\begin{equation}\label{eq-generator22} N^{\sf F}_{A}(M)=\sum_{n\in\N}N^{\sf F}_{(A\cap E_{n})\setminus (A\cap (E_{0}\cup\cdots\cup E_{n-1}))}(M)\;\;\;\mbox{for all}\;M\in C^{\sf F}(S). \end{equation}
By (\ref{eq-generator2}) the maps $N^{\sf F}_{(A\cap E_{n})\setminus (A\cap (E_{0}\cup\cdots\cup E_{n-1}))}$ are $\mathcal{C}^{\sf F}(\mathcal{E})$-measurable and (\ref{eq-generator22}) implies that so is $N^{\sf F}_{A}$.\end{proof}

\begin{proof}[\normalfont\bf Proof of Theorem \ref{sigma-theorem}.]
Since $\pi_{1}$ and $\pi_{2}$ are $\sigma$-finite on $\mathcal{E}$, there exists a common covering $\{E_{n}\,|\,n\in\N\}\subseteq \mathcal{E}$ with $S=\bigcup_{n\in\N}E_{n}$ and $|\pi_{i}\cap E_{n}|<\infty$ for all $n\in\N$, $i=1,2$. Therefore, defining ${\sf F}:=\{E_{n}\,|\,n\in\N\}$, we can consider $\pi_{1}$ and $\pi_{2}$ to be maps $\Om\to C^{\sf F}(S)$.
Hence we have \begin{equation}\label{dkn} P_{\pi_{i}}(A)=P_{\pi_{i}}(A\cap C^{\sf F}(S))\;\; \mbox{for all}\; A\in\mathcal{C}(\mathcal{S})\;\mbox{and}\;i=1,2,\end{equation} and it suffices to show $P_{\pi_{1}}=P_{\pi_{2}}$ on the trace $\sigma$-field $\mathcal{C}^{\sf F}(\mathcal{S})=\mathcal{C}(\mathcal{S})\cap C^{\sf F}(S)$.\\ We deduce from Theorem \ref{theo-generator2} that the sets of the form
\[ \{N^{\sf F}_{A_{1}}=k_{1},\ldots, N^{\sf F}_{A_{n}}=k_{n}\},\;\;n\in\N,A_{j}\in\mathcal{E},k_{j}\in\N_{0}\cup\{\infty\},\]
are a $\cap$-stable generator of $\mathcal{C}^{\sf F}(\mathcal{S})$. By (\ref{eq-theosigma}) and (\ref{dkn}) and by \[ \{N^{\sf F}_{A_{1}}=k_{1},\ldots, N^{\sf F}_{A_{n}}=k_{n}\}=\{N_{A_{1}}=k_{1},\ldots, N_{A_{n}}=k_{n}\}\cap C^{\sf F}(S)\] we conclude that $P_{\pi_{1}}$ and $P_{\pi_{2}}$ coincide on that generator. Therefore the uniqueness theorem of measure theory yields $P_{\pi_{1}}=P_{\pi_{2}}$.\end{proof}


\subsection{Hitting Functions and Determining Classes} 
\setcounter{theorem}{0}

Let $\pi:\Om\to C(S)$ be a cr-set. We will use the following notation for the hitting function of $\pi$:
\[ T_{\pi}:\mathcal{S}\rightarrow [0,1],\:A\mapsto P(\pi\cap A\neq \emptyset). \]
If $\Delta\in\mathcal{S}\otimes\mathcal{S}$, then Corollary \ref{prop2-3} and the uniqueness theorem of measure theory ensure that the law of $\pi$ is uniquely
determined by $T_{\pi}$. Since this fact is the central motivation for this section, we put it down in a proposition:

\begin{prop}\label{prop-central}
 Let $\pi_{1},\pi_{2}:\Om\to C(S)$ be two cr-sets. If $\Delta\in\mathcal{S}\otimes\mathcal{S}$ and if $T_{\pi_{1}}(A)=T_{\pi_{2}}(A)$ for all $A\in\mathcal{S}$, then $\pi_{1}$ and $\pi_{2}$ are equal in law on $(C(S),\mathcal{C}(\mathcal{S}))$.
\end{prop}

In the following we are going to study those functions $T_{\pi}$, in order to find preferably simple classes $\mathcal{E}\subsetneq \mathcal{S}$ and conditions such that $T_{\pi}$ is determined by its values on $\mathcal{E}$. 
Proposition \ref{prop-central} then shows that such classes $\mathcal{E}$ determine also the law of $\pi$.

We start by listing some basic but important properties of hitting functions. Let $T:\mathcal{S}\to [0,1]$ be the hitting function of a cr-set, then 

\begin{enumerate}
 \item[$\bullet$] $T(\emptyset)=0$,
 \item[$\bullet$] $T$  is {\it monotone}, i.e. $A\subseteq B$ implies $T(A)\leq T(B)$ for all $A,B\in\mathcal{S}$,
 \item[$\bullet$] $T$ is {\it $\sigma$-subadditive}, i.e. $T(\bigcup_{n\in\N}A_{n})\leq \sum_{n\in\N}T(A_{n})$ for all $A_{n}\in\mathcal{S}$,
 \item[$\bullet$] $T$ is {\it continuous from below}, i.e. $T(A_{n})\to T(A)$, if $A_{n}\in \mathcal{S}$ with $A_{n}\subseteq A_{n+1}$ for all $n\in\nz$ and $A=\bigcup_{n\in\nz}A_{n}$.
\end{enumerate}

A hitting function $T$, satisfying $T(A_{n})\to T(A)$ if $A_{n}\in\mathcal{S}$ with $A_{n+1}\subseteq A_{n}$ for all $n\in\nz$ and $A=\bigcap_{n\in\nz}A_{n}$,
is called {\it continuous from above} or simply {\it continuous}. In general $T$ is not continuous from above. However, we can always find a determining class for continuous hitting functions by the following proposition.

\begin{prop} \label{prop2-4}
 Let $\mathcal{R}\subseteq \mathcal{S}$ be a ring with $\sigma(\mathcal{R})=\mathcal{S}$ and let $T_{1},T_{2}:\mathcal{S}\to [0,\infty]$ be two continuous hitting functions. If there exist $E_{n}\in\mathcal{R}$, $n\in\nz$, covering $S$ and if $T_{1}(A)=T_{2}(A)$ for all $A\in \mathcal{R}$, then $T_{1}=T_{2}$. 
\end{prop}

\begin{proof}[\normalfont\bfseries Proof] The continuity yields that $\mathcal{M}:=\{A\in\mathcal{S}\,|\,T_{1}(A)=T_{2}(A)\}$ is a monotone class. By $\mathcal{R}\subseteq \mathcal{M}$ and the monotone class theorem \cite[\S 6 Theorem B.]{Halmos},  $\mathcal{M}$ includes the smallest $\sigma$-ring containing $\mathcal{R}$, which we call $\Sigma$. Since $S=\bigcup_{n\in\nz}E_{n}$ with $E_{n}\in\mathcal{R}$, it follows that $S\in\Sigma$ and $\Sigma$ is a $\sigma$-field. From $\sigma(\mathcal{R})=\mathcal{S}$ it follows that $\Sigma=\mathcal{S}$, which leads to $\mathcal{M}=\mathcal{S}$.\end{proof}

\begin{prop} \label{prop2-5}
Let $\pi:\Om\to C(S)$ be a finite cr-set. Then $T_{\pi}$ is continuous from above.
 \end{prop}

\begin{proof}[\normalfont\bfseries Proof]
 Suppose $A_{n}\in\mathcal{S}$ with $A_{n+1}\subseteq A_{n}$ for all $n\in\N$ and $A=\bigcap_{n\in\N}A_{n}$.
 The relation $\{\pi\cap A\neq \emptyset\}\subseteq \bigcap_{n\in\N}\{\pi\cap A_{n}\neq\emptyset\}$ holds
 for any cr-set $\pi$. In order to show the other inclusion, let $\om$ be an element of $\bigcap_{n\in\N}\{\pi\cap A_{n}\neq\emptyset\}$.
 Since $\pi(\om)$ is finite, there exists $x\in\pi(\om)$ lying in infinitely many $A_{n}$. Otherwise $\pi(\om)\cap A_{n}\neq \emptyset$ would hold for only finitely many $n\in\N$ -- a contradiction. Because the sets $A_{n}$ are decreasing, it follows that $x\in A$ and hence $\om\in\{\pi\cap A\neq \emptyset\}$.\\ Therefore $T_{\pi}(A_{n})\to T_{\pi}(A)$ follows from $\{\pi\cap A_{n+1}\neq \emptyset\}\subseteq \{\pi\cap A_{n}\neq \emptyset\}$ for all $n\in\N$ and the continuity of $P$.\end{proof}

Next, we introduce notations for set-systems, that are going to serve as determining classes in the following. For this purpose, let $\emptyset\neq \mathcal{E}$ be a family of subsets of $S$. As usual, $\mathcal{E}_{\sigma}$ denotes the collection of all countable unions of $\mathcal{E}$-sets and $\mathcal{E}_{\delta}$ denotes the collection of all countable intersections respectively. Note that $\emptyset$ as the empty union shall not necessarily belong to $\mathcal{E}_{\sigma}$ such that $\emptyset\in \mathcal{E}_{\sigma}$ implies $\emptyset \in \mathcal{E}$. The same shall apply to $S$ with respect to $\mathcal{E}_{\delta}$. In addition to that, define $\mathcal{E}^{c}:=\{E^{c}\,|\,E\in\mathcal{E}\}$, $\mathcal{E}_{ext}:=\mathcal{E}_{\sigma}$ and $\mathcal{E}_{int}:= (\mathcal{E}^{c})_{\delta}=\{\bigcap_{n\in\N}E_{n}^{c}\,|\,E_{n}\in\mathcal{E}\}$. 

In the context of determining classes, we imagine
$\mathcal{E}_{int}$ to serve as a set-system for approximation of $\mathcal{S}$-sets from within, whereas $\mathcal{E}_{ext}$ is supposed to approximate $\mathcal{S}$-sets from the outside. Again we list some basic statements. 

\smallskip

\begin{enumerate}
 \item[$\bullet$] $(\mathcal{E}_{ext})^{c}=\mathcal{E}_{int}$.
 \item[$\bullet$] $\mathcal{E}\subseteq (\mathcal{E}_{int})_{\sigma}$ if and only if $\mathcal{E}^{c}\subseteq (\mathcal{E}_{ext})_{\delta}$.
 \item[$\bullet$] $\mathcal{E}_{ext}$ is closed under formation of countable unions.
 \item[$\bullet$] $\mathcal{E}_{int}$ is closed under formation of countable intersections.
 \item[$\bullet$] If $\mathcal{E}$ is $\cap$-stable, then $\mathcal{E}_{ext}$ is $\cap$-stable and $\mathcal{E}_{int}$ is $\cup$-stable.
\end{enumerate}

\smallskip

It will be an important technical assumption for us in the following that $\mathcal{E}$ is  $\cap$-stable with $\emptyset\in\mathcal{E}$ and $\mathcal{E}\subseteq (\mathcal{E}_{int})_{\sigma}$. Relating to these notions, we give two examples, that we have mainly in mind.

\begin{enumerate}
 \item[(A)] Let $\mathcal{E}$ be a semiring with $S=\bigcup_{n\in\N}E_{n}$ for suitable $E_{n}\in\mathcal{E}$. Then
 $\mathcal{E}$ is $\cap$-stable and satisfies $\mathcal{E}\subseteq (\mathcal{E}_{int})_{\sigma}$ as well as $\emptyset\in\mathcal{E}$.
\item[(B)] Let $(S,\mathcal{G})$ be a topological space, in which all closed sets are countable intersections of open sets, i.e. $\mathcal{G}^{c}\subseteq \mathcal{G}_{\delta}$. For example, this is the case in all metrizable or perfectly normal topological spaces. Now define $\mathcal{E}:=\mathcal{G}$. Then again $\mathcal{E}=\mathcal{E}_{ext}$ is $\cap$-stable and satisfies $\mathcal{E}\subseteq (\mathcal{E}_{int})_{\sigma}$ as well as $\emptyset\in\mathcal{E}$.
\end{enumerate}

\begin{theorem}\label{theo-dc1}
 Let $\mathcal{E}$ be a $\cap$-stable collection of subsets of $S$ generating $\mathcal{S}$ with $\emptyset\in\mathcal{E}$ and $\mathcal{E}\subseteq (\mathcal{E}_{int})_{\sigma}$. If $T$ is a continuous hitting function, then for all $A\in\mathcal{S}$ and $\varepsilon >0$ there exist $F\in \mathcal{E}_{int}$ and $G\in\mathcal{E}_{ext}$ with $F\subseteq A\subseteq G$ and $T(G\setminus F)\leq \varepsilon$.
\end{theorem}

\begin{proof}[\normalfont\bfseries Proof] Let $\mathcal{A}$ be the set of all $A\in\mathcal{S}$ such that for all $\varepsilon >0$ exist $F\in\mathcal{E}_{int}$ and $G\in\mathcal{E}_{ext}$ with $F\subseteq A\subseteq G$ and $T(G\setminus F)\leq \varepsilon$. We are going to show that $\mathcal{A}$ is a $\sigma$-field containing $\mathcal{E}$.\\
Let $E$ be an element of $\mathcal{E}$ and let $\varepsilon>0$. Define $G:=E$. Since $\mathcal{E}\subseteq (\mathcal{E}_{int})_{\sigma}$, there exist $F_{n}\in\mathcal{E}_{int}$ with $E=\bigcup_{n\in\N}F_{n}$. The sets $G\setminus (F_{1}\cup\cdots \cup F_{n})$, $n\in\N$,  are decreasing and their intersection is empty. Therefore, by continuity from above, there exists $N\in\N$ with $T(G\setminus \bigcup_{n=1}^{N}F_{n})\leq \epsilon$. Now defining $F:=\bigcup_{n=1}^{N}F_{n}$, we conclude $E\in\mathcal{A}$.\\
Since $\emptyset\in\mathcal{E}$ and $\mathcal{E}\subseteq (\mathcal{E}_{int})_{\sigma}$, it follows that $\emptyset\in\mathcal{E}_{int}$ and, consequently, $\emptyset \in\mathcal{A}$. For all sets $A\in\mathcal{A}$ and $F\in\mathcal{E}_{int}$, $G\in\mathcal{E}_{ext}$ with $F\subseteq A\subseteq G$, we observe that $G^{c}\subseteq A^{c}\subseteq F^{c}$ holds as well as $G^{c}\in\mathcal{E}_{int}$, $F^{c}\in\mathcal{E}_{ext}$ and $G\setminus F=F^{c}\setminus G^{c}$. From this we deduce that $\mathcal{A}$ is closed under formation of complements.\\
Finally, let $A_{n}\in \mathcal{A}$, $n\in\N$, and let $\varepsilon >0$. For all $n\in\N$ there exist $F_{n}\in\mathcal{E}_{int}$ and $G_{n}\in\mathcal{E}_{ext}$ with $F_{n}\subseteq A_{n}\subseteq G_{n}$ and $T(G_{n}\setminus F_{n})\leq \varepsilon/2^{n+1}$. Defining $G:=\bigcup_{n\in\N}G_{n}$ and $K:=\bigcup_{n\in\N}F_{n}$, it follows $G\setminus K \subseteq \bigcup_{n\in\N}(G_{n}\setminus F_{n})$
and by subadditivity we obtain \[ T(G\setminus K)\leq \sum_{n\in\N}T(G_{n}\setminus F_{n}) \leq \varepsilon /2. \] Again the sets $G\setminus (F_{1}\cup\cdots \cup F_{n})$, $n\in\N$, are decreasing, their intersection is $G\setminus K$ and therefore continuity and monotonicity yield the existence of $N\in\N$ with
$T(G\setminus \bigcup_{n=1}^{N}F_{n})-T(G\setminus K)\leq \varepsilon /2$. Defining $F:=\bigcup_{n=1}^{N}F_{n}$, we get $F\subseteq \bigcup_{n\in\N}A_{n}\subseteq G$ as well as \[ T(G\setminus F)=\left( T(G\setminus F) - T(G\setminus K)\right) + T(G\setminus K)\leq \varepsilon.\]
Hence we conclude $\bigcup_{n\in\N}A_{n}\in\mathcal{A}$.\end{proof}

\begin{corollary}\label{coro-dc1}
 In the situation of Theorem \ref{theo-dc1}, the continuous hitting function $T$ satisfies
 \[ T(A)=\sup\{ T(F)\,|\,F\in\mathcal{E}_{int},\,F\subseteq A\}=\inf\{T(G)\,|\,G\in\mathcal{E}_{ext},\,A\subseteq G\}\] for all $A\in\mathcal{S}$.
\end{corollary}

\begin{proof}[\normalfont\bfseries Proof]
 Since the function $T$ is monotone, it follows that $T(A)$ is an upper bound of $\{T(F)\,|\,F\in\mathcal{E}_{int},\,F\subseteq A\}$. Let $\varepsilon >0$ and let $F\in\mathcal{E}_{int}$, $G\in\mathcal{E}_{ext}$ with $F\subseteq A\subseteq G$ and $T(G\setminus F)\leq \varepsilon$. Then subadditivity and monotonicity of $T$ yield \[ T(A)-T(F)\leq T(A\setminus F)\leq T(G\setminus F)\leq \varepsilon \] and we conclude that $T(A)$ is the least upper bound of the mentioned set. The same argument shows that $T(A)=\inf\{T(G)\,|\,G\in\mathcal{E}_{ext},\,A\subseteq G\}$. \end{proof}

\begin{lemma}\label{prop-dc}
 Let $\mathcal{E}$ be a $\cap$-stable collection of subsets of $S$ generating $\mathcal{S}$ with $\emptyset\in\mathcal{E}$ and $\mathcal{E}\subseteq (\mathcal{E}_{int})_{\sigma}$. If $\pi$ is a constructive cr-set, then for all $A\in\mathcal{S}$ there exist $K\in (\mathcal{E}_{int})_{\sigma}$, $L\in(\mathcal{E}_{ext})_{\delta}$ with $K\subseteq A\subseteq L$ and $T_{\pi}(L\setminus K)=0$.
\end{lemma}

\begin{proof}[\normalfont\bfseries Proof]
 Let $\pi=\bigcup_{k\in\N}\pi_{k}$ with finite cr-sets $\pi_{k}$ and let $A\in\mathcal{S}$. By Theorem \ref{theo-dc1}, for $\varepsilon>0$ and $k\in\N$ there exist $F_{k}\in\mathcal{E}_{int},G_{k}\in \mathcal{E}_{ext}$ with $F_{k}\subseteq A\subseteq G_{k}$ and $T_{\pi_{k}}(G_{k}\setminus F_{k})\leq \varepsilon/2^{k}$. Defining $L_{\varepsilon}:=\bigcap_{k\in\N}G_{k}$ and $K_{\varepsilon}:=\bigcup_{k\in\N}F_{k}$, we have $\{\pi\cap (L_{\varepsilon}\setminus K_{\varepsilon})\neq\emptyset\}=\bigcup_{k\in\N}\{\pi_{k}\cap (L_{\varepsilon}\setminus K_{\varepsilon})\neq\emptyset\}$ and obtain by $\sigma$-subadditivity of $P$  \[ T_{\pi}(L_{\varepsilon}\setminus K_{\varepsilon})\leq \sum_{k\in\N}T_{\pi_{k}}(L_{\varepsilon}\setminus K_{\varepsilon}).\]
 From this we deduce by monotonicity of $T_{\pi_{k}}$ and by $L_{\varepsilon}\setminus K_{\varepsilon}\subseteq G_{k}\setminus F_{k}$ that
  \[ T_{\pi}(L_{\varepsilon}\setminus K_{\varepsilon})\leq \sum_{k\in\N}T_{\pi_{k}}(G_{k}\setminus F_{k})\leq \varepsilon.\]
 
Finally, observe that $L:=\bigcap_{n\in\N}L_{1/n}$ is an element of $(\mathcal{E}_{ext})_{\delta}$ and $K:=\bigcup_{n\in\N}K_{1/n}$ is an element of $(\mathcal{E}_{int})_{\sigma}$, satisfying $T_{\pi}(L\setminus K)\leq T_{\pi}(L_{1/n}\setminus K_{1/n})\leq 1/n$ for all $n\in\N$ as well as $K\subseteq A \subseteq L$. This proves the claim. \end{proof}

\begin{theorem}\label{theo-dc2}
  Let $\mathcal{E}$ be a $\cap$-stable collection of subsets of $S$ generating $\mathcal{S}$ with $\emptyset\in\mathcal{E}$ and $\mathcal{E}\subseteq (\mathcal{E}_{int})_{\sigma}$. If $\pi$ is a constructive cr-set, then its hitting function satisfies \[ T_{\pi}(A)=\sup\{ T_{\pi}(F)\,|\,F\in\mathcal{E}_{int},\,F\subseteq A\}\] for all $A\in\mathcal{S}$.
\end{theorem}

\begin{proof}[\normalfont\bfseries Proof]
By monotonicity $T_{\pi}(A)$ is an upper bound of $\{ T_{\pi}(F)\,|\,F\in\mathcal{E}_{int},\,F\subseteq A\}$ for all $A\in\mathcal{S}$.
Applying Lemma \ref{prop-dc} on $A\in\mathcal{S}$, there exists $K\in (\mathcal{E}_{int})_{\sigma}$ with $K\subseteq A$ and $T_{\pi}(A\setminus K)=0$.  Subadditivity and monotonicity of $T_{\pi}$ therefore yield \[ T_{\pi}(K) \leq T_{\pi}(A)\leq T_{\pi}(A\setminus K) + T_{\pi}(K) = T_{\pi}(K)\]
and we conclude $T_{\pi}(A)=T_{\pi}(K)$. Since $K\in (\mathcal{E}_{int})_{\sigma}$ and since $\mathcal{E}_{int}$ is closed under formation of finite unions, there exist $F_{n}\in\mathcal{E}_{int}$ with $K=\bigcup_{n\in\N}F_{n}$ and $F_{n}\subseteq F_{n+1}$ for all $n\in\N$. So by continuity from below of $T_{\pi}$, it follows that $T_{\pi}(A)$ is the least upper bound of $\{ T_{\pi}(F)\,|\,F\in\mathcal{E}_{int},\,F\subseteq A\}$. \end{proof}

\begin{corollary} \label{coro-dc2}
 Let $\mathcal{E}$ be a $\cap$-stable collection of subsets of $S$ generating $\mathcal{S}$ with $\emptyset\in\mathcal{E}$ and $\mathcal{E}\subseteq (\mathcal{E}_{int})_{\sigma}$. Let $\pi_{1}$ and $\pi_{2}$ be cr-sets satisfying $T_{\pi_{1}}(F)=T_{\pi_{2}}(F)$ for all $F\in\mathcal{E}_{int}$. If $\pi_{1}$ is constructive, then $T_{\pi_{1}}(A)\leq T_{\pi_{2}}(A)$ for all $A\in\mathcal{S}$.
 \end{corollary}

\begin{proof}[\normalfont\bfseries Proof]
 By coincidence of $T_{\pi_{1}}$ and $T_{\pi_{2}}$ on $\mathcal{E}_{int}$ and by Theorem \ref{theo-dc2}, it follows that
\begin{eqnarray*} T_{\pi_{1}}(A) & = & \sup\{T_{\pi_{1}}(F)\,|\,F\in\mathcal{E}_{int},\,F\subseteq A\} \\ & = &  \sup\{T_{\pi_{2}}(F)\,|\,F\in\mathcal{E}_{int},\,F\subseteq A\}\end{eqnarray*} for all $A\in\mathcal{S}$. Since $T_{\pi_{2}}$ is monotone, $T_{\pi_{2}}(A)$ is an upper bound of $\{T_{\pi_{2}}(F)\,|\,F\in\mathcal{E}_{int},\,F\subseteq A\}$. Therefore we conclude $T_{\pi_{1}}(A)\leq T_{\pi_{2}}(A)$.\end{proof}

\begin{proof}[\normalfont\bfseries Proof of Theorem \ref{theo-unique-ring}] Since $\pi_{1}$ is $\sigma$-finite on the ring $\mathcal{R}$, there exist $R_{n}\in\mathcal{R}$, $n\in\N$, covering $S$ such that $R_{n}\subseteq R_{n+1}$ and $\pi_{1}\cap R_{n}$ is finite for all $n\in\N$. Note that $T_{\pi\cap R_{n}}(A)=T_{\pi}(A\cap R_{n})$ for any cr-set $\pi$ and for any $A\in\mathcal{S}$.\\ We are going to show that $T_{\pi_{1}\cap R_{n}}=T_{\pi_{2}\cap R_{n}}$ for all $n\in\N$. Since $T_{\pi_{1}\cap R_{n}}(A)=T_{\pi_{2}\cap R_{n}}(A)$ for all $A\in \mathcal{R}$ and by continuity from below, the functions $T_{\pi_{1}\cap R_{n}}$ and $T_{\pi_{2}\cap R_{n}}$ coincide on $\mathcal{R}_{ext}$. By Proposition \ref{prop2-5}, the function $T_{\pi_{1}\cap R_{n}}$ is continuous and we can apply Corollary \ref{coro-dc1} with $\mathcal{E}=\mathcal{R}$ resulting in \begin{eqnarray}  T_{\pi_{1}\cap R_{n}}(A) & = & \inf\{ T_{\pi_{1}\cap R_{n}}(G)\,|\,G\in\mathcal{R}_{ext}, A\subseteq G\}\nonumber\\ & = & \inf\{ T_{\pi_{2}\cap R_{n}}(G)\,|\,G\in\mathcal{R}_{ext}, A\subseteq G\}\label{eq4} \end{eqnarray} for all $A\in\mathcal{S}$. By monotonicity $T_{\pi_{2}\cap R_{n}}(A)$ is a lower bound of the set $\{ T_{\pi_{2}\cap R_{n}}(G)\,|\,G\in\mathcal{R}_{ext}, A\subseteq G\}$ and therefore (\ref{eq4}) yields \begin{equation}\label{eq5} T_{\pi_{2}\cap R_{n}}(A)\leq T_{\pi_{1}\cap R_{n}}(A)\end{equation} for all $A\in\mathcal{S}$. Now we are able to conclude that $T_{\pi_{2}\cap R_{n}}$ is continuous as well. Let $A_{m}\in\mathcal{S}$ with $A_{m+1}\subseteq A_{m}$ for all $m\in\N$ and $A:=\bigcap_{m\in\N}A_{m}$. Then the sets $A_{m}\setminus A$ are decreasing and their intersection is empty. Hence by subadditivity we obtain \[ T_{\pi_{2}\cap R_{n}}(A_{m})-T_{\pi_{2}\cap R_{n}}(A)\leq T_{\pi_{2}\cap R_{n}}(A_{m}\setminus A)\stackrel{\mbox{\footnotesize (\ref{eq5})}}{\leq} T_{\pi_{1}\cap R_{n}}(A_{m}\setminus A).\]
Therefore $T_{\pi_{1}\cap R_{n}}(A_{m}\setminus A)\to 0$ as $m\to\infty$ implies $T_{\pi_{2}\cap R_{n}}(A_{m})\to T_{\pi_{2}\cap R_{n}}(A)$ and $T_{\pi_{2}\cap R_{n}}$ is continuous. Since both hitting functions are continuous and coincide on $\mathcal{R}$, Proposition \ref{prop2-4} yields $T_{\pi_{1}\cap R_{n}}=T_{\pi_{2}\cap R_{n}}$.\\
Finally, by continuity from below, we conclude that \[ T_{\pi_{1}}(A)=\lim_{n\to\infty}T_{\pi_{1}}(A\cap R_{n})=\lim_{n\to\infty}T_{\pi_{2}}(A\cap R_{n})=T_{\pi_{2}}(A) \] for all $A\in\mathcal{S}$. Proposition \ref{prop-central} then proves the claim of Theorem \ref{theo-unique-ring}. \end{proof}

\begin{proof}[\normalfont\bfseries Proof of Theorem \ref{theo-unique-closed}]
By Proposition \ref{prop2-1}, we have $\Delta\in\mathcal{S}\otimes\mathcal{S}$. Therefore Proposition \ref{prop-central} yields that equality in law follows from $T_{\pi_{1}}=T_{\pi_{2}}$. Furthermore, choosing $\mathcal{E}$ as the collection of open sets in $S$, we see that $\mathcal{E}_{int}$ is the family of closed sets. $\mathcal{E}$ is $\cap$-stable, satisfies $\emptyset\in\mathcal{E}$ and generates the $\sigma$-field $\mathcal{S}$. Since in metric spaces all open sets are countable unions of closed sets, we get $\mathcal{E}\subseteq (\mathcal{E}_{int})_{\sigma}$ as well.\\
a) By Theorem \ref{theo-dc2}, condition (\ref{eq-neu1}) of Theorem \ref{theo-unique-closed} implies $T_{\pi_{1}}=T_{\pi_{2}}$. As mentioned above, this is sufficient.\\
b) Again we are going to show $T_{\pi_{1}}(A)=T_{\pi_{2}}(A)$ for all $A\in\mathcal{S}$. So let us fix $A\in\mathcal{S}$.
By Lemma \ref{prop-dc}, there exists $K\in (\mathcal{E}_{int})_{\sigma}$ with $K\subseteq A$ and $T_{\pi_{2}}(A\setminus K)=0$. By (\ref{eq-tuc2}) we get $T_{\pi_{1}}(A\setminus K)=0$ as well. From this point on, we follow the arguments in the proof of Theorem \ref{theo-dc2}. Those yield $T_{\pi_{1}}(A)=T_{\pi_{1}}(K)$, $T_{\pi_{2}}(A)=T_{\pi_{2}}(K)$ and the existence of $F_{n}\in\mathcal{E}_{int}$ with $K=\bigcup_{n\in\N}F_{n}$ and $F_{n}\subseteq F_{n+1}$ for all $n\in\N$. Therefore (\ref{eq-neu1}) and continuity from below imply $T_{\pi_{1}}(K)=T_{\pi_{2}}(K)$ resulting in $T_{\pi_{1}}(A)=T_{\pi_{2}}(A)$.\\
c) In metric spaces all closed sets are $G_{\delta}$-sets. Therefore the assumption in c) implies (\ref{eq-neu1}). Now applying b), it is sufficient to show (\ref{eq-tuc2}). Note that $(\mathcal{E}_{ext})_{\delta}$ is the collection of $G_{\delta}$-sets.
Fix $A\in\mathcal{S}$ with $T_{\pi_{2}}(A)=0$. Then, by Lemma \ref{prop-dc}, there exists a $G_{\delta}$-set $L$ with $A\subseteq L$ and $T_{\pi_{2}}(L\setminus A)=0$. Subadditivity yields \[ 0\leq T_{\pi_{2}}(L)\leq T_{\pi_{2}}(L\setminus A) + T_{\pi_{2}}(A) = 0\] and we conclude $T_{\pi_{1}}(L)=0$, since $T_{\pi_{1}}$ and $T_{\pi_{2}}$ coincide on $G_{\delta}$-sets. Hence $A\subseteq L$ implies $T_{\pi_{1}}(A)=0$, and (\ref{eq-tuc2}) has been shown. \end{proof}


\subsection{Constructiveness} \label{const}
\setcounter{theorem}{0}

Going over the proof of Kingman's existence theorem for Poisson processes in \cite{kingman}, we can see that all Poisson processes constructed there are countable unions of finite cr-sets. This observation was one reason to introduce Definition \ref{defi-constr}, which we will study in the following. We put this observation down in a remark:

\begin{remark} \label{theo-existence}
Let $\Delta\in\mathcal{S}\otimes\mathcal{S}$. For all $n\in\N$ let $\mu_{n}$ be a finite measure on $(S,\mathcal{S})$ with $\mu_{n}(\{x\})=0$ for all $x\in S$. Then there exists a constructive Poisson process $\pi:\Om\to C(S)$ with intensity $\mu:=\sum_{n\in\N}\mu_{n}$.
\end{remark}

Furthermore, Kingman proved in \cite[Section 2.2]{kingman} that the Cartesian product of two finite cr-sets $\pi_{1},\pi_{2}: (\Om,\mathcal{F})\to (C(S),\mathcal{C}(\mathcal{S}))$ is a finite cr-set $(\Om,\mathcal{F})\to (C(S\times S),\mathcal{C}(\mathcal{S}\otimes\mathcal{S}))$. In the case of $\Delta\in\mathcal{S}\otimes\mathcal{S}$, he concluded that $\pi_{1}\cap \pi_{2}$ is a finite cr-set and it follows easily that $\pi_{1}\setminus \pi_{2}$ and $\pi_{1}\cup\pi_{2}$ are finite cr-sets as well. We start this section by deducing some useful properties of constructive maps from this.

\begin{prop} \label{prop-constructive2}
If $\Delta\in\mathcal{S}\otimes\mathcal{S}$, then every constructive map is the disjoint union of countably many finite cr-sets. Therefore, under measurability of the diagonal, constructive maps are countable random sets.
\end{prop}

\begin{proof}[\normalfont\bfseries Proof] Let $\tau: \Om\to C(S)$ be a constructive map and let $\pi_{n}$ be finite cr-sets with $\tau(\om)=\bigcup_{n\in\N}\pi_{n}(\om)$ for all $\om\in\Om$.\\
It follows that $\tau=\pi_{1} \cup (\pi_{2}\setminus \pi_{1}) \cup (\pi_{3}\setminus (\pi_{1}\cup \pi_{2}))\cup \cdots$ is the disjoint union of the sets $\pi_{1}$ and $\pi_{n}\setminus (\pi_{1}\cup\cdots\cup \pi_{n-1})$, which are finite cr-sets, as we mentioned above. Therefore 
\[ N_{A}(\tau)=N_{A}(\pi_{1})+\sum_{n=2}^{\infty}N_{A}(\pi_{n}\setminus (\pi_{1}\cup\cdots\cup \pi_{n-1})) \]
and the maps $N_{A}(\tau)$ are $\mathcal{F}$-$\mathcal{B}([0,\infty ])$ measurable for all $A\in\mathcal{S}$. By definition of $\mathcal{C}(\mathcal{S})$, this yields the $\mathcal{F}$-$\mathcal{C}(\mathcal{S})$ measurability of $\tau$.\end{proof}

By similar arguments, we can conclude that constructive maps are closed under formation of various set operations.
We will not give a proof.

\begin{prop} \label{prop-constructive3}
 Suppose $\tau_{n}:\Om\to C(S)$, $n\in\N$, are constructive maps. Then $\tau_{1}\times \tau_{2}$ and $\bigcup_{n\in\N}\tau_{n}$ are constructive. If $\Delta\in\mathcal{S}\otimes\mathcal{S}$, then $\tau_{1}\setminus \tau_{2}$ and $\bigcap_{n\in\N}\tau_{n}$ are constructive as well.
\end{prop}

Proposition \ref{prop-constructive3} was another reason to introduce the notion of constructiveness. We believe that in general cr-sets fail to be closed under formation of those set operations -- even under the assumption of $\Delta\in\mathcal{S}\otimes\mathcal{S}$.
 
There do exist non-constructive cr-sets. Kendall gives in \cite{kendall} the following example, which is an adaptation of \cite[Example 5]{himmel2}: 

\begin{example}
Let $(S,\mathcal{S})$ be the real numbers $\R$ equipped with its Borel-$\sigma$-field $\mathcal{B}(\R)$ and let $\Om:=C(S)\setminus\{\emptyset\}$ and $\mathcal{F}:=\{D\cap \Om\,|\,D\in\mathcal{C}(\mathcal{S})\}$. The map $\pi :\Om\to C(S)$, $\om\mapsto \om$ is $\mathcal{F}$-$\mathcal{C}(\mathcal{S})$ measurable and not constructive. If $\pi$ was constructive, then Theorem \ref{theo-constructive1} would yield the existence of a measurable $X:\Om\to \R$ with $X(\om)\in\pi(\om)$ for all $\om\in\Om$. Since $\R^{\N}\to \Om$, $(x_{n})_{n\in\N}\mapsto \{x_{n}\,|\,n\in\N\}$ is $\mathcal{B}(\R^{\N})$-$\mathcal{F}$ measurable, we would conclude that $g:\R^{\N}\to \R$, $(x_{n})_{n\in\N}\mapsto X(\{x_{n}\,|\,n\in\N\})$ is $\mathcal{B}(\R^{\N})$-$\mathcal{B}(\R)$ measurable. But according to \cite[Corollary 2]{blackwell} such $g$ does not exist.
\end{example}

Next, we are going to prove Theorem \ref{theo-constructive1}, which will be prepared in the following. As in Section 2, we need nested sequences of partitions of $S$ -- a concept presented e.g. in \cite{daley}:

 Let $E_{1},E_{2},\ldots$ be a countable family of sets in $\mathcal{S}$ separating points of $S$. Define inductively
\begin{eqnarray*}
  & \bullet & Z_{1,1}:=E_{1}\;\:\mbox{and}\;\:Z_{1,2}:=E_{1}^{c}, \\
  & \bullet & Z_{n,2k-1}:= Z_{n-1,k}\cap E_{n}\:\;\mbox{and}\:\;Z_{n,2k}:=Z_{n-1,k}\cap E_{n}^{c}
\end{eqnarray*}
for $n\in\N$, $n\geq 2$ and $k\in\{1,\ldots,2^{n-1}\}$. Observe that, for every single $n\in\N$, the sets $Z_{n,1},\ldots,Z_{n,2^{n}}$ are a partition of $S$. Now let $\emptyset \neq M \subseteq S$ be a finite subset of $S$. For all $n\in\N$ define $k(n)=k(n,M)$ to be the least element $k\in\{1,\ldots,2^{n}\}$ such that $M\cap Z_{n,k}\neq \emptyset$.\\
By definition of $k(n)$ and $Z_{n,k}$, it follows that $Z_{n,k(n)}\supseteq Z_{n+1,k(n+1)}$ for all $n\in\N$. Since $M$ is finite, we conclude that there exists $x\in S$ such that $x\in M\cap Z_{n,k(n)}$ for all $n\in\N$.
Finally, because the sets $E_{1},E_{2},\ldots$ are separating points, we deduce that 
\begin{equation} \label{eq6} \{x\} = \bigcap_{n\in\N} M\cap Z_{n,k(n)}.\end{equation}

 For a fixed countable family of sets in $\mathcal{S}$ separating points of $S$ we introduce the following notations:
\begin{enumerate}
 \item[$\bullet$]  If $\emptyset\neq M\subseteq S$ is a finite set, define $X(M)$ to be the element $x\in S$ in (\ref{eq6}).
 \item[$\bullet$] If $\pi:\Om\to C(S)$ is finite and if $Y:\Om\to S$ is a map, define $X[\pi,Y]:\Om\to S$ by \[ X[\pi,Y](\om):=\left\{\begin{array}{ll} Y(\om) & \mbox{if}\;\;\om\in\{\pi=\emptyset\}, \\ X(\pi(\om)) & \mbox{if}\;\;\om\in\{\pi\neq\emptyset\}. \end{array}\right.\]
\end{enumerate}

\begin{lemma} \label{propA1} Fix a countable family of sets in $\mathcal{S}$ separating points of $S$. 
 Let $\pi$ be a finite cr-set and let $Y:(\Om,\mathcal{F})\to (S,\mathcal{S})$ be measurable. Then $X[\pi,Y]$ is $\mathcal{F}$-$\mathcal{S}$ measurable and satisfies
$X[\pi,Y]\in \pi$ on $\{\pi\neq \emptyset\}$ as well as $X[\pi,Y]=Y$ on $\{\pi=\emptyset\}$. 
\end{lemma}

\begin{proof}[\normalfont\bfseries Proof]
 We have $\{X[\pi,Y]\in A\}=\{\pi=\emptyset, Y\in A\}\cup \{\pi\neq\emptyset, X(\pi)\in A\}$ for all $A\in\mathcal{S}$. Therefore it is sufficient to show that $\{\pi\neq\emptyset, X(\pi)\in A\}\in\mathcal{F}$, which follows from
\[ \{\pi\neq\emptyset, X(\pi)\in A\}= \bigcap_{n\in\N}\bigcup_{k=1}^{2^{n}}\left(\{\pi\cap Z_{n,k}\cap A\neq \emptyset\}\cap \bigcap_{j=1}^{k-1}\{ \pi\cap Z_{n,j}=\emptyset\}\right). \]
The rest is obvious. \end{proof}

\begin{theorem}\label{theoA1}
 Let $\Delta\in\mathcal{S}\otimes\mathcal{S}$ and let $\pi$ be a finite cr-set. Then there exist measurable
 $X_{n}:(\Om,\mathcal{F})\to (S,\mathcal{S})$, $n\in\N$, such that $\pi=1_{\{\pi\neq\emptyset\}}\{X_{1},X_{2},\ldots\}$.
\end{theorem}

\begin{proof}[\normalfont\bfseries Proof]
 Fix a countable family of sets in $\mathcal{S}$ separating points of $S$ and choose any measurable map $Y:(\Om,\mathcal{F})\to (S,\mathcal{S})$. Now define inductively for all $n\in\N$, $n\geq 2$:
\begin{eqnarray*}
  & \bullet & X_{1}:=X[\pi,Y]\;\:\mbox{and}\;\:\pi_{1}:=\pi\setminus\{X_{1}\}, \\
  & \bullet & X_{n}:=X[\pi_{n-1},X_{1}]\:\;\mbox{and}\:\;\pi_{n}:=\pi_{n-1}\setminus\{X_{n}\}.
\end{eqnarray*}
By Lemma \ref{propA1} and by induction, it follows that all $X_{n}$ are measurable and all $\pi_{n}$ are finite cr-sets. Remember that in the case of $\Delta\in\mathcal{S}\otimes\mathcal{S}$ finite cr-sets are closed under formation of set-theoretic differences. Furthermore, we have $X_{n}\in \pi_{n-1}$ on $\{\pi_{n-1}\neq \emptyset\}$ and $X_{n}=X_{1}$ on $\{\pi_{n-1}=\emptyset\}$ for all $n\geq 2$ as well as $X_{1}\in \pi$ on $\{\pi\neq \emptyset\}$.\\
Therefore, if $\om\in \{\pi\neq \emptyset\}$ with $|\pi(\om)|=k$ and $k\in\N$, we conclude that $\pi(\om)=\{X_{1}(\om),\ldots,X_{k}(\om)\}$ and $X_{n}(\om)=X_{1}(\om)$ for all $n>k$. \end{proof}

Choosing a metric on $S$ by Proposition \ref{prop2-1} and applying \cite[Theorem 5.4]{himmel1}, we can achieve a very similar result. However, the difference lies in the Borel-measurability of the selectors in
\cite[Theorem 5.4]{himmel1}, whereas the $X_{n}$ in Theorem \ref{theoA1} are $\mathcal{F}$-$\mathcal{S}$ measurable.

\begin{proof}[\normalfont\bfseries Proof of Theorem \ref{theo-constructive1}] Let $\pi=\bigcup_{n\in\N}\pi_{n}$ with finite cr-sets $\pi_{n}$. Then $\{\pi\neq \emptyset\}$ is the disjoint union of the sets $\{\pi_{1}=\pi_{2}=\cdots =\pi_{n-1}=\emptyset, \pi_{n}\neq\emptyset\}$, $n\in\N$. Let $Y:(\Om,\mathcal{F})\to (S,\mathcal{S})$ be any measurable map and define $X:\Om\to S$ by
\[ X(\om):=\left\{\begin{array}{ll} Y(\om) & \mbox{if}\;\;\om\in\{\pi=\emptyset\}, \\ X[\pi_{n},Y](\om) & \mbox{if}\;\;\om\in\{\pi_{1}=\cdots =\pi_{n-1}=\emptyset, \pi_{n}\neq\emptyset\},\,n\in\N. \end{array}\right.\]
 By Lemma \ref{propA1}, it follows that $X$ is $\mathcal{F}$-$\mathcal{S}$ measurable as well as $X(\om)\in\pi(\om)$ for all $\om\in\{\pi\neq \emptyset\}$. Applying Theorem \ref{theoA1}, for all $n\in\N$ there exist random variables $Y_{n,k}:(\Om,\mathcal{F})\to (S,\mathcal{S})$, $k\in\N$, with $\pi_{n}=1_{\{\pi_{n}\neq\emptyset\}}\{Y_{n,1},Y_{n,2},\ldots\}$. Now define
\[ X_{n,k}(\om):=\left\{\begin{array}{ll} X(\om) & \mbox{if}\;\;\om\in\{\pi_{n}=\emptyset\}, \\ Y_{n,k}(\om) & \mbox{if}\;\;\om\in\{\pi_{n}\neq \emptyset\} \end{array}\right.\] for all $n,k\in\N$. Then all $X_{n,k}$ are $\mathcal{F}$-$\mathcal{S}$ measurable and we have \[\pi=1_{\{\pi\neq\emptyset\}}\{X_{n,k}\,|\,(n,k)\in\N^{2}\}.\]
Since $\N^{2}$ is countable, the proof is complete. \end{proof}

Finally, we prove Theorem \ref{theo-decomp} and Theorem \ref{theo-inde-incre} in this section. For both proofs we need the following lemma and the following proposition.

\begin{lemma}\label{lemma-zorn}
Let $\mathfrak{X}\neq \emptyset$ be any set and let $\mathcal{P}$ be a non-empty collection of subsets of $\mathfrak{X}$, satisfying the following condition:
\begin{equation} \mbox{All subcollections} \,\;\mathcal{E}\subseteq\mathcal{P}\; \mbox{with disjoint elements are countable}.\label{eq-zorn}\end{equation} Then there exists a countable family of disjoint sets $A_{i}\in\mathcal{P}$, $i\in I$, such that $(\bigcup_{i\in I}A_{i})^{c}$ contains no set $A\in\mathcal{P}$ with $A\neq\emptyset$.
\end{lemma}

In the context of Boolean algebras condition (\ref{eq-zorn}) is called ,,countable chain condition`` -- see \cite[Section 14]{Halmos2}. Lemma \ref{lemma-zorn} is a straightforward application of Zorn's lemma. For the reader's convenience, we provide a proof.

\begin{proof}[\normalfont\bfseries Proof of Lemma \ref{lemma-zorn}]
 Define \[ M:=\{ \mathcal{E}\subseteq \mathcal{P}\;|\;\mathcal{E}\;\mbox{is countable and has disjoint elements}\}. \]
The set $M\neq \emptyset$ is partially ordered by inclusion. Let $K\subseteq M$ be a totally ordered subset of $M$ and define
$\mathcal{E}_{K}:=\bigcup_{\mathcal{E}\in K}\mathcal{E}$. Since $K$ is totally ordered, the elements of $\mathcal{E}_{K}$ have to be disjoint. Therefore condition (\ref{eq-zorn}) implies that $\mathcal{E}_{K}$ is countable. Hence $\mathcal{E}_{K}\in M$ and $\mathcal{E}_{K}$ is an upper bound of $K$. Thus we have shown that every totally ordered subset of $M$ has an upper bound.\\
Consequently,  Zorn's lemma yields the existence of a maximal element $\mathcal{E}=\{A_{i}\,|\,i\in I\}$ in $M$. Suppose $(\bigcup_{i\in I}A_{i})^{c}$ contained a nonempty set $A\in\mathcal{P}$. Then $\mathcal{E}\cup\{A\}\in M$ and $\mathcal{E}\subsetneq \mathcal{E}\cup\{A\}$ would contradict the maximality of $\mathcal{E}$. Therefore such $A$ cannot exist.\end{proof}

\begin{prop} \label{prop-zorn1} Let $\Delta\in\mathcal{S}\otimes\mathcal{S}$ and let $\pi:\Om \to C(S)$ be a constructive cr-set, then $\mathcal{P}:=\{ A\in\mathcal{S}\,|\, T_{\pi}(A)>0 \}$ satisfies  (\ref{eq-zorn}).
 \end{prop}

\begin{proof}[\normalfont\bfseries Proof] By Theorem \ref{theo-constructive1}, there exist random variables $X_{1},X_{2},\ldots$ with $\pi=1_{\{\pi\neq \emptyset\}}\{X_{1},X_{2},\ldots \}$. Hence  $\{\pi\cap A\neq\emptyset\}\subseteq \bigcup_{n\in\N}\{ X_{n}\in A\}$ for all $A\in\mathcal{S}$ and for any $\mathcal{E}\subseteq \mathcal{P}$ we conclude
\begin{eqnarray*}  \mathcal{E} & = & \bigcup_{n\in\N}\left\{A\in\mathcal{E}\,|\, P(X_{n}\in A)>0\right\}\\
 			       & = & \bigcup_{n\in\N}\bigcup_{m\in\N}\left\{A\in\mathcal{E}\,|\,P(X_{n}\in A)>1/m\right\}.
\end{eqnarray*} If the elements of $\mathcal{E}$ are disjoint, the sets $\{A\in\mathcal{E}\,|\,P(X_{n}\in A)>1/m\}$ are finite for all $n,m\in\N$ and, consequently, $\mathcal{E}$ is countable.\end{proof}

\begin{proof}[\normalfont\bfseries Proof of Theorem \ref{theo-decomp}]
 First, we prove the existence of $F$. By Proposition \ref{prop-zorn1}, the collection $\mathcal{P}:=\{ A\in\mathcal{S}\,|\, T_{\pi}(A)>0\:\mbox{and}\:P(N_{A}(\pi)=\infty)=0\}$ satisfies (\ref{eq-zorn}). Applying Lemma \ref{lemma-zorn}, we therefore get countably many disjoint $A_{i}\in\mathcal{P}$, $i\in I$, such that \begin{equation}P(\pi\cap A\neq \emptyset)=0\;\;\mbox{or}\;\;P(N_{A}(\pi)=\infty)>0\;\;\mbox{for all}\;\;A\subseteq (\bigcup_{i\in I}A_{i})^{c}.\label{eq-decomp}\end{equation} Define $F:=\bigcup_{i\in I}A_{i}\in\mathcal{S}$ and $E:=\bigcap_{i\in I}\{N_{A_{i}}(\pi)<\infty\}\in\mathcal{F}$. Since $A_{i}\in\mathcal{P}$ for all $i\in I$, we conclude $P(E)=1$. Hence $1_{E}\pi \cap F$ is a $\sigma$-finite version of $\pi\cap F$. Furthermore, it follows from (\ref{eq-decomp}) that $\pi\cap F^{c}$ satisfies condition (ii).
 
Finally, let $F_{1},F_{2}\in\mathcal{S}$ satisfy (i) and (ii). Since $\pi\cap F_{1}$ possesses a $\sigma$-finite version, there exist $A_{n}\in\mathcal{S}$, $n\in\N$, with $P(N_{A_{n}}(\pi\cap F_{1})=\infty)=0$ and $\bigcup_{n\in\N}A_{n}=S$. Now, we deduce from $N_{A_{n}\cap F_{1}}(\pi\cap F_{2}^{c})\leq N_{A_{n}}(\pi\cap F_{1})$ that $P(N_{A_{n}\cap F_{1}}(\pi\cap F_{2}^{c})=\infty)=0$ holds. Since $F_{2}$ satisfies (ii), it follows further that $P((\pi\cap F_{2}^{c})\cap (A_{n}\cap F_{1})\neq\emptyset)=0$. Therefore $S=\bigcup_{n\in\N}A_{n}$ yields
\[ P(\pi\cap (F_{1}\setminus F_{2})\neq \emptyset)=P(\bigcup_{n\in\N}\{(\pi\cap F_{2}^{c})\cap (A_{n}\cap F_{1})\neq \emptyset\})=0.\] The same argument shows that $P(\pi\cap(F_{2}\setminus F_{1})\neq\emptyset)=0$ and we conclude $P(\pi\cap(F_{1}\vartriangle F_{2})\neq\emptyset)=0$. 
\end{proof}

\begin{prop} \label{prop-zorn2}Let $\Delta\in\mathcal{S}\otimes\mathcal{S}$ and let $\pi:\Om \to C(S)$ be a constructive cr-set with independent increments. Then there exist finite measures $\mu_{n}:\mathcal{S}\to [0,\infty)$, $n\in\N$, with $-\log(1-T_{\pi})=\sum_{n\in\N}\mu_{n}$.
\end{prop}

\begin{proof}[\normalfont\bfseries Proof] Since $\pi$ has independent increments, we obtain \begin{eqnarray*}-\log\left(1-T_{\pi}(A\cup B)\right) & = &-\log\left(P(\{\pi\cap A=\emptyset\}\cap\{\pi\cap  B=\emptyset\})\right)\\ & = & -\log\left(P(\pi\cap A=\emptyset)P(\pi\cap B=\emptyset)\right)\\ & = & -\log\left(1-T_{\pi}(A)\right)-\log\left(1-T_{\pi}(B)\right), \end{eqnarray*}
if $A,B\in\mathcal{S}$ are disjoint. From this and from continuity from below of $T_{\pi}$ we deduce that $\mu:=-\log(1-T_{\pi})$ is a measure on $(S,\mathcal{S})$. Note that 
\begin{equation}\label{eq-zorn2} \mu(A)=0 \Leftrightarrow T_{\pi}(A)=0 \;\;\;\;\mbox{and}\;\;\;\;\mu(A)=\infty \Leftrightarrow T_{\pi}(A)=1 
\end{equation}
holds for all $A\in\mathcal{S}$.
Applying Theorem \ref{theo-constructive1}, there exist random variables $X_{1},X_{2},\ldots$  with $\pi=1_{\{\pi\neq \emptyset\}}\{X_{1},X_{2},\ldots \}$. Therefore we are able to define the finite measures $\mu_{n,k}:\mathcal{S}\to [0,\infty)$, $\mu_{n,k}(A):=P(X_{k}\in A,\pi\neq \emptyset)$, $n,k\in\N$. Note that $\mu_{n,k}$ does not really depend on $n$. Nevertheless, we define it that way. From (\ref{eq-zorn2}) and 
\[ \{X_{l}\in A,\pi\neq\emptyset\}\subseteq \{\pi\cap A\neq \emptyset\}\subseteq \bigcup_{k\in\N}\{X_{k}\in A,\pi\neq \emptyset\}\] for all $l\in\N$, $A\in\mathcal{S}$ we conclude \begin{equation}\label{eq-zorn3} \mu(A)=\sum_{(n,k)\in\N^{2}}\mu_{n,k}(A)\;\:\;\mbox{for all}\;\: A\in\mathcal{S}\;\: \mbox{with}\;\: \mu(A)\in\{0,\infty\}. \end{equation}
Now define $\mathcal{P}:=\{A\in\mathcal{S}\,|\,0<T_{\pi}(A)<1\}$. The collection $\mathcal{P}$ satisfies (\ref{eq-zorn}) by Proposition \ref{prop-zorn1}. Note that (\ref{eq-zorn2}) implies $\mathcal{P}=\{A\in\mathcal{S}\,|\,0<\mu(A)<\infty\}$.  Finally, we are going to consider the following two cases. 

If $\mathcal{P}=\emptyset$, we obtain $\mu(A)\in\{0,\infty\}$ for all $A\in\mathcal{S}$ and (\ref{eq-zorn3}) yields the statement of Proposition \ref{prop-zorn2}. 

In the case of $\mathcal{P}\neq\emptyset$, we apply Lemma \ref{lemma-zorn} and obtain a countable family of disjoint sets $A_{i} \in\mathcal{P}$, $i\in I$, such that $\mu(A)\in\{0,\infty\}$ for all $A\in \mathcal{S}$ with $A\subseteq (\bigcup_{i\in I}A_{i})^{c}$. Define $D:=(\bigcup_{i\in I}A_{i})^{c}$ and $\mu_{i}(A):=\mu(A\cap A_{i})$ for all $i\in I$. The measures $\mu_{i}$ are finite and it follows that\[ \mu(A)=\sum_{i\in I}\mu_{i}(A) + \mu(A\cap D)\stackrel{\mbox{\footnotesize (\ref{eq-zorn3})}}{=}\sum_{i\in I}\mu_{i}(A)+ \sum_{(n,k)\in\N^{2}}\mu_{n,k}(A\cap D)\] holds for all $A\in\mathcal{S}$. Therefore $\mu$ is the sum of countably many finite measures and Proposition \ref{prop-zorn2} is proven. \end{proof}

\begin{proof}[\normalfont\bfseries Proof of Theorem \ref{theo-inde-incre}] By Proposition \ref{prop-zorn2}, we have \[ P(\pi\cap A=\emptyset)=e^{-\sum_{n\in\N}\mu_{n}(A)}\] for all $A\in\mathcal{S}$. Under the assumptions of Theorem \ref{theo-inde-incre}, we have $\mu_{n}(\{x\})=0$ for all $x\in S$. Therefore the existence theorem in \cite{kingman} together with Proposition \ref{prop-central} yield Theorem \ref{theo-inde-incre}.
\end{proof}

\bigskip

\noindent {\bf Discussion of Example \ref{example2}.}

\smallskip

\noindent First, remember that Lebesgue's density theorem \cite[Subsection 7.12]{rudin} states  
\[ \lim_{h\to +0} \frac{\lambda(A\cap (z-h,z+h))}{2h}= 1 \]
for every Borel-set $A\subseteq \R$ and for $\lambda$-almost all $z\in A$. By comparing to the harmonic series, we deduce from this \begin{equation}\label{eqex1} \sum_{n=1}^{\infty}\lambda\left(A\cap (z-\frac{1}{n},z+\frac{1}{n})\right)=\infty\end{equation}
for $\lambda$-almost all $z\in A$. Secondly, Fubini's theorem and independence yield
\begin{equation}\label{eqex2} P(N_{A}(\tau)=\infty)=\int_{\R} P(N_{A-z}(\pi)=\infty)\,dP_{Z}(z)\end{equation}
with $A-z:=\{a-z\,|\,a\in A\}$ and $P_{Z}$ denoting the law of $Z$. Since $\pi$ is a Poisson process with intensity $\mu$, we have
\begin{equation}\label{eqex3} P(N_{A-z}(\pi)=\infty)=\left\{\begin{array}{cc} 1 & \mbox{for}\;\mu(A-z)=\infty, \\  0 & \mbox{for}\;\mu(A-z)<\infty. \end{array}\right.\end{equation}
If $P_{Z}$ is equivalent to $\lambda$, then (\ref{eqex1}) holds also for $P_{Z}$-almost all $z\in A$. Together with (\ref{eqex2}),(\ref{eqex3}) and $\mu(A-z)=\sum_{n\in\N}\lambda(A\cap (z-\frac{1}{n},z+\frac{1}{n}))$, this yields
\[ P(N_{A}(\tau)=\infty)\geq P_{Z}(A)>0\] for all Borel-sets $A\subset \R$ with $\lambda(A)>0$.
Finally, it is easy to check that $P(N_{A}(\tau)=\infty)<1$ for bounded $A$.



\begin{acknowledgement}
   I thank Martin M\"ohle for his advice and for fruitful discussions. Furthermore, i would like to thank the referee for his careful reading of the manuscript and for his valuable comments.
\end{acknowledgement}


\end{document}